\documentclass[12pt]{amsart}
\usepackage{amscd,amsmath,amsthm,amssymb,amsfonts,enumerate,hyperref}

\def\NZQ{\mathbb}               
\def\NN{{\NZQ N}}

\def\ZZ{{\NZQ Z}}
\def\RR{{\NZQ R}}

\def\Dc{{\mathcal D}}
\def\Pc{{\mathcal P}}
\def\Qc{{\mathcal Q}}
\def\Mc{{\mathcal M}}

\def\Bc{{\mathcal B}}

\def\ab{{\mathbf a}}

\def\eb{{\mathbf e}}

\def\opn#1#2{\def#1{\operatorname{#2}}} 
	\opn\chara{char} \opn\length{\ell} \opn\pd{pd} \opn\rk{rk}
	\opn\projdim{proj\,dim} \opn\injdim{inj\,dim} \opn\rank{rank}
	\opn\depth{depth} \opn\grade{grade} \opn\height{height}
	\opn\embdim{emb\,dim} \opn\codim{codim}
	\opn\Cl{Cl}
	
	\opn\Tr{Tr} \opn\bigrank{big\,rank}
	\opn\superheight{superheight}\opn\lcm{lcm}
	\opn\trdeg{tr\,deg}
	\opn\rdeg{rdeg}
	\opn\reg{reg} \opn\lreg{lreg} \opn\ini{in} \opn\lpd{lpd}
	\opn\size{size} \opn\sdepth{sdepth}
	\opn\link{link}\opn\fdepth{fdepth}\opn\lex{lex}
	\opn\tr{tr}
	\opn\type{type}
	\opn\gap{gap}
	\opn\arithdeg{arith-deg}
	\opn\revlex{revlex}
	\opn\div{div} \opn\Div{Div} \opn\cl{cl} \opn\Cl{Cl}
	\opn\Spec{Spec} \opn\Supp{Supp} \opn\supp{supp} \opn\Sing{Sing}
	\opn\Ass{Ass} \opn\Min{Min}\opn\Mon{Mon}
	\opn\Ann{Ann} \opn\Rad{Rad} \opn\Soc{Soc}
	\opn\Im{Im} \opn\Ker{Ker} \opn\Coker{Coker} \opn\Am{Am}
	\opn\Hom{Hom} \opn\Tor{Tor} \opn\Ext{Ext} \opn\End{End}
	\opn\Aut{Aut} \opn\id{id}
	
	\opn\nat{nat}
	\opn\pff{pf}
	\opn\Pf{Pf} \opn\GL{GL} \opn\SL{SL} \opn\mod{mod} \opn\ord{ord}
	\opn\Gin{Gin} \opn\Hilb{Hilb}\opn\sort{sort}
	\opn\PF{PF}\opn\Ap{Ap}
	\opn\mult{mult}
	\opn\bight{bight}
	\opn\div{div}
	\opn\Div{Div}
	\opn\aff{aff}
	\opn\relint{relint} \opn\st{st}
	\opn\lk{lk} \opn\cn{cn} \opn\core{core} \opn\vol{vol}  \opn\inp{inp}
	\opn\nilpot{nilpot}
	\opn\link{link} \opn\star{star}\opn\lex{lex}\opn\set{set}
	\opn\width{wd}
	\opn\Fr{F}
	\opn\QF{QF}
	\opn\G{G}
	\opn\type{type}\opn\res{res}
	\opn\conv{conv}
	\opn\Int{Int}
	\opn\Deg{Deg}
	\opn\Sym{Sym}
	\opn\Con{Con}
	\opn\gr{gr}
	
	\def\pot#1#2{#1[\kern-0.28ex[#2]\kern-0.28ex]}

	\opn\dirlim{\underrightarrow{\lim}}
	\opn\inivlim{\underleftarrow{\lim}}
	\let\to=\rightarrow
	
	\def\Implies{\ifmmode\Longrightarrow \else
		\unskip${}\Longrightarrow{}$\ignorespaces\fi}
	\def\implies{\ifmmode\Rightarrow \else
		\unskip${}\Rightarrow{}$\ignorespaces\fi}
	\def\iff{\ifmmode\Longleftrightarrow \else
		\unskip${}\Longleftrightarrow{}$\ignorespaces\fi}

	\let\:=\colon
	\newtheorem{Theorem}{Theorem}[section]
	\newtheorem{Lemma}[Theorem]{Lemma}

	\theoremstyle{definition}

	\newtheorem{Example}[Theorem]{Example}

    \makeatletter
\@namedef{subjclassname@2020}{%
  \textup{2020} Mathematics Subject Classification}
\makeatother

\begin{document}

\title[Gorenstein polytopes]{Bounded powers of edge ideals: Gorenstein polytopes}

\author[T.~Hibi]{Takayuki Hibi}
\author[S.~A.~ Seyed Fakhari]{Seyed Amin Seyed Fakhari}

\address{(Takayuki Hibi) Department of Pure and Applied Mathematics, Graduate School of Information Science and Technology, Osaka University, Suita, Osaka 565--0871, Japan}
\email{hibi@math.sci.osaka-u.ac.jp}
\address{(Seyed Amin Seyed Fakhari) Departamento de Matem\'aticas, Universidad de los Andes, Bogot\'a, Colombia}
\email{s.seyedfakhari@uniandes.edu.co}



\begin{abstract}
Let $S=K[x_1, \ldots,x_n]$ denote the polynomial ring in $n$ variables over a field $K$ and $I(G) \subset S$ the edge ideal of a finite graph $G$ on $n$ vertices.  Given a vector $\mathfrak{c}\in\NN^n$ and an integer $q\geq 1$, we denote by $(I(G)^q)_{\mathfrak{c}}$ the ideal of $S$ generated by those monomials belonging to $I(G)^q$ whose exponent vectors are componentwise bounded above by $\mathfrak{c}$. Let $\delta_{\mathfrak{c}}(I(G))$ denote the largest integer $q$ for which $(I(G)^q)_{\mathfrak{c}}\neq (0)$.  Since $(I(G)^{\delta_{\mathfrak{c}}(I)})_{\mathfrak{c}}$ is a polymatroidal ideal, it follows that its minimal set of monomial generators is the set of bases of a discrete polymatroid $\Dc(G,\mathfrak{c})$.  In the present paper, a classification of Gorenstein polytopes of the form $\conv(\Dc(G,\mathfrak{c}))$ is studied.
\end{abstract}

\maketitle

\thispagestyle{empty}

\section{Introduction} \label{sec1}

Let $S=K[x_1, \ldots,x_n]$ denote the polynomial ring  in $n$ variables over a field $K$ with $n \geq 3$.  If $u \in S$ is a monomial, then $M_{\leq u}$ stands for the set of those monomials $w \in S$ which divide $u$.  In particular, $1 \in M_{\leq u}$ and $u \in M_{\leq u}$.  Let $G$ be a finite graph on the vertex set $V(G)=\{x_1, \ldots, x_n\}$, where $n \geq 3$, with no loop, no multiple edge and no isolated vertex, and $E(G)$ the set of edges of $G$.  Recall that the edge ideal of $G$ is the ideal $I(G) \subset S$ which is generated by those $x_ix_j$ with $\{x_i, x_j\} \in E(G)$. Let $\ZZ_{>0}$ denote the set of positive integers. Given a vector $\mathfrak{c}=(c_1,\ldots,c_n)\in (\ZZ_{>0})^n$ and an integer $q\geq 1$, we denote by $(I(G)^q)_\mathfrak{c}$ the ideal of $S$ generated by those monomials $x_1^{a_1}\cdots x_n^{a_n} \in I(G)^q$ with $a_i \leq c_i$ for each $i=1, \ldots, n$.  Let $\delta_{\mathfrak{c}}(I(G))$ denote the biggest integer $q$ for which $(I(G)^q)_\mathfrak{c} \neq (0)$.  Then $(I(G)^{\delta_{\mathfrak{c}}(I(G))})_\mathfrak{c}$ is a polymatroidal ideal (\cite[Theorem 4.3]{HSF}).  Let $\Bc(G,\mathfrak{c})$ denote the minimal set of monomial generators of $(I(G)^{\delta_{\mathfrak{c}}(I(G))})_\mathfrak{c}$.   Also, set $\Mc(G,\mathfrak{c}):= \{ M_{\leq u} : u \in \Bc(G,\mathfrak{c})\}$ and
\[
\Dc(G,\mathfrak{c}):=\{(a_1, \ldots, a_d) \in \ZZ^d : x_1^{a_1}\cdots x_n^{a_n} \in \Mc(G,\mathfrak{c})\}.
\]
The unit coordinate vectors $\eb_1, \ldots, \eb_n$ of $\RR^n$ together with the origin $(0,\ldots, 0) \in \RR^d$ belong to $\Dc(G,\mathfrak{c})$.  Since $(I(G)^{\delta_{\mathfrak{c}}(I(G))})_\mathfrak{c}$ is a polymatroidal ideal, it follows from \cite[Theorem 2.3]{HH_discrete} that $\Dc(G,\mathfrak{c})$ is a discrete polymatroid \cite[Definition 2.1]{HH_discrete}.  Now, we introduce $\conv(\Dc(G,\mathfrak{c})) \subset \RR^n$, which is the convex hull of $\Dc(G,\mathfrak{c})$ in $\RR^n$.  It then follows from \cite[Theorem 3.4]{HH_discrete} that $\conv(\Dc(G,\mathfrak{c}))$ is a polymatroid \cite[p.~240]{HH_discrete}.  

Let $2^{[n]}$ denote the set of subsets of $[n]:=\{1, \ldots, n\}$.  The {\em ground set rank function} \cite[p.~243]{HH_discrete} $\rho_{(G,\mathfrak{c})} : 2^{[n]} \to \ZZ_{>0}$ of $\conv(\Dc(G,\mathfrak{c}))$ is defined by setting
\[
\rho_{(G,\mathfrak{c})}(X) = \max\left\{\sum_{i \in X}a_i : x_1^{a_1}\cdots x_n^{a_n} \in \Bc(G,\mathfrak{c})\right\}
\]
for $\emptyset \neq X \subset [n]$ together with $\rho_{(G,\mathfrak{c})}(\emptyset)=0$. A nonempty subset $A\subset [n]$ is called {\em $\rho_{(G,\mathfrak{c})}$-closed} if for any $B\subset [n]$ with $A \subsetneq B$, one has $\rho_{(G,\mathfrak{c})}(A) < \rho_{(G,\mathfrak{c})}(B)$. A nonempty subset $A\subset [n]$ is called {\em $\rho_{(G,\mathfrak{c})}$-separable} if there exist nonempty subsets $A'$ and $A''$ of $[n]$ with $A = A' \cup A''$ and $A' \cap A'' = \emptyset$ for which $\rho_{(G,\mathfrak{c})}(A) = \rho_{(G,\mathfrak{c})}(A') + \rho_{(G,\mathfrak{c})}(A'')$.  

Our original motivation to organize the present paper is to classify the Gorenstein polytopes of the form $\conv(\Dc(G,\mathfrak{c}))$.  First, recall what Gorenstein polyotopes are. A convex polytope $\Pc \subset \RR^n$ is called a {\em lattice polytope} if each of whose vertices belongs to $\ZZ^n$.  A {\em reflexive polytope} is a lattice polytope $\Pc \subset \RR^n$ of dimension $n$ for which the origin of $\RR^n$ belongs to the interior of $\Pc$ and the {\em dual polytope} 
$$
\Pc^\vee = \{ (x_1, \ldots, x_n) \in \RR^n : \sum_{i=1}^{n} x_iy_i \leq 1, \forall (y_1, \ldots, y_n) \in \Pc \}
$$
of $\Pc$ is again a lattice polytope.  A lattice polytope $\Pc \subset \RR^n$ of dimension $n$ is called {\em Gorenstein} if there is an integer $\delta > 0$ together with a vector $\ab \in \ZZ^n$ for which $\delta \Pc - \ab$ is a reflexive polytope (\cite{Hcombinatorica}). The following lemma \cite[Theorem 7.3]{HH_discrete} has a key role in this paper.

\begin{Lemma}[\cite{HH_discrete}]
\label{criterion}
The lattice polytope $\conv(\Dc(G,\mathfrak{c})) \subset \RR^n$ is Gorenstein if and only if there is an integer $k >0$ for which
\[
\rho_{(G,\mathfrak{c})}(A)=\frac{1}{\, k \,}(|A|+1)
\]
for all $\rho_{(G,\mathfrak{c})}$-closed and $\rho_{(G,\mathfrak{c})}$-inseparable subsets $A \subset [n]$.    
\end{Lemma}

After recalling basic materials on finite graphs in Section \ref{sec2}, and on grand set rank functions in Section \ref{sec3}, we classify Gorenstein polytopes of the form $\conv(\Dc(G,\mathfrak{c}))$ arising from complete graphs and cycles (Section \ref{sec4}), complete bipartite graphs (Section \ref{sec5}), paths (Section \ref{sec6}), regular bipartite graphs (Section \ref{sec7}), whiskered graphs (Section \ref{sec8}) and Cohen--Macaulay 
Cameron--Walker graphs (Section \ref{sec9}).  

Let ${\Qc}_n \subset \RR^n$ be the standard unit cube whose vertices are $(\varepsilon_1, \ldots, \varepsilon_n)$ with each $\varepsilon_i \in \{0,1\}$ and $\Qc_n' := 2 \Qc_n - (1, \ldots, 1) \subset \RR^n$, whose vertices are $(\pm 1, \ldots, \pm 1) \in \RR^n$.  Since $\Qc'_n$ is reflexive, both ${\Qc}_n$ and $\Qc_n'+(1,\ldots,1)$ are Gorenstein.  In addition to ${\Qc}_n$ and $\Qc_n'+(1,\ldots,1)$, several Gorenstein polytopes of the form $\conv(\Dc(G,\mathfrak{c}))$ arise.  See Examples \ref{ham}, \ref{FGHIJ} and \ref{pathBBBBB}.  A Gorenstein polytope of the form $\conv(\Dc(G,\mathfrak{c}))$ which is neither ${\Qc}_n$ nor $\Qc_n'+(1,\ldots,1)$ is called {\em exceptional Gorenstein polytope}.  To calssify all exceptional Gorenstein polytopes is reserved for our forthcoming study.

\section{Finite graphs} \label{sec2}
Let $n \geq 3$ and $G$ a finite graph on the vertex set $V(G)=\{x_1, \ldots, x_n\}$ with no loop, no multiple edge and no isolated vertex.  Let $E(G)$ be the set of edges of $G$. 

We say that two vertices $x_i, x_j \in V(G)$ are {\em adjacent} in $G$ if $\{x_i,x_j\} \in E(G)$.  In addition, $x_j$ is called a {\em neighbor} of $x_i$. The set of neighbors of $x_i$ is denoted by $N_G(x_i)$. The cardinality of $N_G(x_i)$ is the {\em degree} of $x_i$, denoted by ${\deg}_G(x_i)$. We say that $e \in E(G)$ is {\em incident} to $x_i \in V(G)$ if $x_i \in e$. A subgraph $H$ of $G$ is called an {\em induced subgraph} if for any $x_i, x_j\in V(H)$, one has $\{x_i, x_j\}\in E(H)$ if and only if $\{x_i, x_j\}\in E(G)$. A subgraph $H$ of $G$ is called a {\em spanning subgraph} if $V(H)=V(G)$. A subset $A \subset V(G)$ is called {\em independent} if $\{x_i, x_j\} \not\in E(G)$ for all $x_i, x_j \in A$ with $i \neq j$.
 
The {\em complete graph} $K_n$ is the finite graph on $[n]$ whose edges are those $\{x_i,x_j\}$ with $1 \leq i < j \leq n$.  

The {\em complete bipartite graph} $K_{n,m}$ is the finite graph on $$\{x_1, \ldots,x_n\}\sqcup\{x_{n+1}, \ldots, x_{n+m}\}$$ whose edges are those $\{x_i,x_j\}$ with $1 \leq i \leq n$ and $n+1 \leq j \leq n+m$.

A {\em matching} of $G$ is a subset $M \subset E(G)$ for which $e \cap e' = \emptyset$ for $e, e' \in M$ with $e \neq e'$. The size of the largest matching of $G$ is called the {\em matching number} of $G$, denoted by ${\rm match}(G)$. A {\em perfect matching} of $G$ is a matching $M$ of $G$ with $\cup_{e \in M}e=V(G)$.  

The {\em cycle} of length $n$ is the finite graph $C_n$ on $\{x_1, \ldots, x_n\}$ whose edges are $$\{x_1,x_2\},\{x_2,x_3\},\ldots, \{x_{n-1},x_n\},\{x_1,x_n\}.$$  

A finite graph $G$ on $n$ vertices is called {\em Hamiltonian} if $G$ contains $C_n$ after a suitable relabeling of the vertices.  

In the polynomial ring $S = K[x_1, \ldots, x_n]$, unless there is a misunderstanding, for an edge $e = \{x_i, x_j\}$, we employ the notation $e$ instead of the monomial $x_ix_j\in S$.  For example, if $e_1 = \{x_1, x_2\}$ and $e_2 = \{x_2, x_5\}$, then $e_1^2e_2 = x_1^2x_2^3x_5$.

\section{Basic facts on ground set rank functions} \label{sec3}
We summarize basic behavior on the ground set rank function of $\conv(\Dc(G,\mathfrak{c}))$.  Let $n \geq 3$ and $G$ a finite graph on $V(G)=\{x_1, \ldots, x_n\}$.  Also, let $\mathfrak{c}=(c_1, \ldots, c_n)\in (\ZZ_{>0})^n$.    

\begin{Lemma} \label{rhomin}
Let $i \in [n]$.  One has
\[
\rho_{(G,\mathfrak{c})}(\{i\})=\min\big\{c_i, \sum_{x_k\in N_G(x_i)}c_k\big\}.
\]
\end{Lemma}

\begin{proof}
Clearly one has $\rho_{(G,\mathfrak{c})}(\{i\})\leq \min\big\{c_i, \sum_{x_k\in N_G(x_i)}c_k\big\}$.  Now, assume that $$\rho_{(G,\mathfrak{c})}(\{i\})< \min\big\{c_i, \sum_{x_k\in N_G(x_i)}c_k\big\}.$$  Set $\delta:=\delta_{\mathfrak{c}}(I(G))$. Let $u\in \Bc(G,\mathfrak{c})$ be a monomial with ${\rm deg}_{x_i}(u)=\rho_{(G,\mathfrak{c})}(\{i\})$. Then $u$ can be written as $u=e_1\cdots e_{\delta}$, where $e_1, \ldots, e_{\delta}$ are edges of $G$. If there is a vertex $x_p\in N_G(x_i)$ with ${\rm deg}_{x_p}(u)< c_p$, then $(x_ix_p)u\in (I(G)^{\delta+1})_{\mathfrak{c}}$ which is a contradiction. Thus, for each vertex $x_p\in N_G(x_i)$, one has ${\rm deg}_{x_p}(u)=c_p$. Since
$${\rm deg}_{x_i}(u)=\rho_{(G,\mathfrak{c})}(\{i\})< \sum_{x_p\in N_G(x_i)}c_p=\sum_{x_p\in N_G(x_i)}{\rm deg}_{x_p}(u),$$
in the representation of $u$ as $u=e_1\cdots e_{\delta}$, there is an edge, say $e_1$ which is incident to a vertex $x_p\in N_G(x_i)$ but not to $x_i$. Hence, $e_1=\{x_p,x_{p'}\}$, for some vertex $x_{p'}\neq x_i$. Then
$$\frac{ux_i}{x_{p'}}=(x_ix_p)e_2\cdots e_{\delta}\in \Bc(G,\mathfrak{c}),$$
and$$
\rho_{(G,\mathfrak{c})}(\{i\})\geq {\rm deg}_{x_1}(ux_i/x_{p'})>{\rm deg}_{x_i}(u)=\rho_{(G,\mathfrak{c})}(\{i\}),$$which is a contradiction.
\end{proof}

\begin{Lemma} \label{noinclusion1}
Suppose that $i \in [n]$ enjoys the property that, for each $k \in [n]$ with $\{x_i,x_k\}\notin E(G)$, one has $N_G(x_k)\nsubseteq N_G(x_i)$.  Then the singleton $\{i\}$ is $\rho_{(G,\mathfrak{c})}$-closed (and $\rho_{(G,\mathfrak{c})}$-inseparable).    
\end{Lemma}

\begin{proof}
To prove the assertion, it is enough to prove that for each $j\in [n]$ with $j\neq i$, the inequality $\rho_{(G,\mathfrak{c})}(\{i,j\})>\rho_{(G,\mathfrak{c})}(\{i\})$ holds. Indeed, let $u\in \Bc(G,\mathfrak{c})$ be a monomial with ${\rm deg}_{x_i}(u)=\rho_{(G,\mathfrak{c})}(\{i\})$. If $u$ is divisible by $x_j$, then the inequality $\rho_{(G,\mathfrak{c})}(\{i,j\})>\rho_{(G,\mathfrak{c})}(\{i\})$ trivially holds. So, suppose that $x_j$ does not divide $u$. Set $\delta:=\delta_{\mathfrak{c}}(IG))$. As $u\in (I(G)^{\delta})_{\mathfrak{c}}$, it can be written as $u=e_1\cdots e_{\delta}$, where $e_1, \ldots, e_{\delta}$ are edges of $G$. As $u$ is divisible by $x_i$, we may assume that $e_1=\{x_i,x_p\}$ for some vertex $x_p$ of $G$. Since $u$ is not divisible by $x_j$, we conclude that $p\neq j$. If $x_i$ and $x_j$ are adjacent in $G$, then$$\frac{ux_j}{x_p}=(x_ix_j)e_2\cdots e_{\delta}\in \Bc(G,\mathfrak{c}).$$Consequently,
\begin{align*}
\rho_{(G,\mathfrak{c})}(\{i,j\})& \geq {\rm deg}_{x_i}(ux_j/x_p)+{\rm deg}_{x_j}(ux_j/x_p)>{\rm deg}_{x_i}(ux_j/x_p)\\ &={\rm deg}_{x_i}(u)=\rho_{(G,\mathfrak{c})}(\{i\}).
\end{align*}
So, assume that $x_i$ and $x_j$ are not adjacent in $G$. By assumption, there is a vertex $x_q\in N_G(x_j)\setminus N_G(x_i)$. If $x_q$ does not divide $u$, then $(x_jx_q)u\in (I(G)^{\delta+1})_{\mathfrak{c}}$ which is a contradiction. Therefore, $x_q$ divides $u$. Hence, we  may assume that $e_{\delta}=\{x_q,x_{q'}\}$, for some vertex $x_{q'}$ of $G$. Since $x_q\notin N_G(x_i)$, one has $q'\neq i$. Note that$$\frac{ux_j}{x_{q'}}=e_1e_2\cdots e_{\delta-1}(x_jx_q)\in \Bc(G,\mathfrak{c}).$$Thus,
\begin{align*}
\rho_{(G,\mathfrak{c})}(\{i,j\})& \geq {\rm deg}_{x_i}(ux_j/x_{q'})+{\rm deg}_{x_j}(ux_j/x_{q'})>{\rm deg}_{x_i}(ux_j/x_{q'})\\ &={\rm deg}_{x_i}(u)=\rho_{(G,\mathfrak{c})}(\{i\}).
\end{align*}
Consequently, $\{i\}$ is $\rho_{(G,\mathfrak{c})}$-closed.
\end{proof}

\begin{Lemma}
\label{noinclusion}
Suppose that $G$ is a connected graph with the property that, if $x_i, x_j \in V(G)$ are nonadjacent, then $N_G(x_i)\nsubseteq N_G(x_j)$.  If $\conv(\Dc(G,\mathfrak{c}))$ is Gorenstein, then either $c_1=\cdots =c_n=1$ or $c_1=\cdots =c_n=2$.
\end{Lemma}

\begin{proof}
It follows from Lemma \ref{noinclusion1} and the assumption that for any $i\in [n]$, the singleton $\{i\}$ is $\rho_{(G,\mathfrak{c})}$-closed (and $\rho_{(G,\mathfrak{c})}$-inseparable).  For each $i\in [n]$, set $\rho_i:=\rho_{(G,\mathfrak{c})}(\{i\}$.  We conclude from Lemma \ref{criterion} that either $\rho_1=\cdots = \rho_n = 1$ or $\rho_1=\cdots = \rho_n = 2$. To complete the proof, we show that $\rho_i=c_i$, for each $i\in [n]$. If $c_i\leq \sum_{x_k\in N_G(x_i)}c_k$, then the assertion follows from Lemma \ref{rhomin}. So, suppose that $c_i> \sum_{x_k\in N_G(x_i)}c_k$. Again using Lemma \ref{rhomin}, we deduce that $\rho_k=c_k$, for each integer $k$ with $x_k\in N_G(x_i)$. Moreover, $\rho_i=\sum_{x_k\in N_G(x_i)}c_k$. Since $G$ is a connected graph on $n\geq 3$ vertices, it follows from the assumption that $x_i$ is not a leaf of $G$. So, there are two distinct vertices $x_{k_1}, x_{k_2}\in N_G(x_i)$.  It follows that$$\rho_i=\sum_{x_k\in N_G(x_i)}c_k\geq c_{k_1}+c_{k_2}=\rho_{k_1}+\rho_{k_2}.$$This is a contradiction, as $\rho_1=\cdots =\rho_n$.
\end{proof}

\section{Complete graphs and cycles} \label{sec4}
In this section, a few examples of Gorenstein polytopes of the form $\conv(\Dc(G,\mathfrak{c}))$ are given and the Gorenstein polytopes arising from complete graphs are classified.  

Let ${\Qc}_n \subset \RR^n$ be the standard unit cube  whose vertices are $(\varepsilon_1, \ldots, \varepsilon_n)$ with each $\varepsilon_i \in \{0,1\}$.  Since the cube $\Qc_n' := 2 \Qc_n - (1, \ldots, 1) \subset \RR^n$, whose vertices are $(\pm 1, \ldots, \pm 1) \in \RR^n$, is reflexive, it follows that $\Qc_n$ is Gorenstein.

\begin{Example}
    \label{perfectmatching}
Let $n \geq 4$ be even and $G$ a finite graph on $V(G)=\{x_1, \ldots, x_n\}$ for which $G$ has a perfect matching.  Let $\mathfrak{c}=(1, \ldots, 1)\in (\ZZ_{>0})^n$.  One has $\delta_{\mathfrak{c}}(I(G))=n/2$ and $\Bc(G,\mathfrak{c}) = \{x_1\cdots x_n\}$.  Since $\rho_{(G,\mathfrak{c})}(X) = |X|$ for $X \subset [n]$, it follows that $X \subset [n]$ is $\rho_{(G,\mathfrak{c})}$-closed and $\rho_{(G,\mathfrak{c})}$-inseparable if and only if $|X| = 1$.  Hence $\conv(\Dc(G,\mathfrak{c}))$ is Gorenstein (Lemma \ref{criterion}).  More precisely, one has $\conv(\Dc(G,\mathfrak{c})) = \Qc_n$.
\end{Example}

If $n \geq 3$ is an odd integer, then the standard unit cube ${\Qc}_n \subset \RR^n$ cannot be of the form $\conv(\Dc(G,\mathfrak{c}))$.  In fact, if $G$ is a finite graph on $V(G)=\{x_1, \ldots, x_n\}$ and ${\Qc}_n = \conv(\Dc(G,\mathfrak{c}))$, then $x_1\cdots x_n \in \Bc(G,\mathfrak{c})$, which is impossible, since the degree of each monomial belonging to $\Bc(G,\mathfrak{c})$ is even.

\begin{Example} \label{ham}
Let $n \geq 3$ and $\mathfrak{c}=(1, \ldots, 1)\in (\ZZ_{>0})^n$.  Let $G$ be a Hamiltonian graph on $V(G)=\{x_1, \ldots, x_n\}$. If $n$ is even, then $G$ has a perfect matching and $\conv(\Dc(G,\mathfrak{c})) = \Qc_n$.

Let $n$ be odd.  One has $\delta_{\mathfrak{c}}(I(G))=(n-1)/2$ and $\Bc(G,\mathfrak{c}) = \{u/x_1, \ldots, u/x_n\}$, where $u=x_1 \cdots x_n$.  One has $\rho_{(G,\mathfrak{c})}([n]) = n-1$ and $\rho_{(G,\mathfrak{c})}(X) = |X|$ for $X \subsetneq [n]$.  Thus $X \subset [n]$ is $\rho_{(G,\mathfrak{c})}$-closed and $\rho_{(G,\mathfrak{c})}$-inseparable if and only if either $|X| = 1$ or $X=[n]$.  It then follows from Lemma \ref{criterion} that $\conv(\Dc(G,\mathfrak{c}))$ is Gorenstein if and only if $n=3$.  When $n=3$, $\conv(\Dc(G,\mathfrak{c})) \subset \RR^3$ is the Gorenstein polytope $\Pc_3 \subset \RR^3$ which is defined by the system of linear inequalities $0 \leq x_i \leq 1$ for $1 \leq i \leq 3$ together with $x_1+x_2+x_3 \leq 2$.
\end{Example}

\begin{Example}
    \label{2...2}
Let $n \geq 3$ and $G$ a finite graph on $V(G)=\{x_1, \ldots, x_n\}$ for which either $G$ has a pefect matching or $G$ is Hamiltonian.  Let $\mathfrak{c} =(2,\ldots,2) \in (\ZZ_{>0})^n$.  One has $\delta_{\mathfrak{c}}(I(G))=n$ and $\Bc(G,\mathfrak{c}) = \{x_1^2\cdots x_n^2\}$.  Thus  $\conv(\Dc(G,\mathfrak{c})) = {\Qc}'_n + (1,\ldots, 1)$, which is Gorenstein.
\end{Example}

\begin{Example}
    \label{cycle}
Let $n \geq 3$ and $G=C_n$ the cycle of length $n$ on $V(G)=\{x_1, \ldots, x_n\}$.  Let $\mathfrak{c} \in (\ZZ_{>0})^n$ and suppose that $\conv(\Dc(C_n,\mathfrak{c}))$ is Gorenstein.  Then either $c_1=\cdots =c_n=1$ or $c_1=\cdots =c_n=2$ (Lemma \ref{noinclusion}).  Let $c_1=\cdots =c_n=2$.  Since $C_n$ is Hamiltonian, one has $\conv(\Dc(G,\mathfrak{c})) = {\Qc}'_n + (1,\ldots, 1)$ (Example \ref{2...2}).

Let $c_1=\cdots =c_n=1$.  If $n$ is even, then $G$ has a perfect matching and $\conv(\Dc(G,\mathfrak{c})) = \Qc_n$ (Example \ref{perfectmatching}).  Let $n$ be odd.  Since $C_n$ is Hamiltonian, it follows that $\conv(\Dc(C_n,\mathfrak{c}))$ is Gorenstein if and only if $n=3$ (Example \ref{ham}).
\end{Example}

We now come to the classification of Gorenstein polytopes arising from complete graphs.
      
\begin{Theorem}
\label{complete}
Let $n \geq 3$ and $K_n$ the complete graph on $V(G)=\{x_1, \ldots, x_n\}$.  The Gorenstein polytopes of the form $\conv(\Dc(K_n,\mathfrak{c}))$, are exactly 
\begin{itemize}
    \item [(i)] ${\Qc}'_n + (1,\ldots, 1)$,
    \item [(ii)] $\Qc_n$ with $n$ even, and
    \item[(iii)] $\Pc_3$ of Example \ref{ham}.
\end{itemize}
\end{Theorem}

\begin{proof}
Suppose that $\conv(\Dc(K_n,\mathfrak{c}))$ is Gorenstein.  One has either $c_1=\cdots =c_n=1$ or $c_1=\cdots =c_n=2$ (Lemma \ref{noinclusion}).  Let $c_1=\cdots =c_n=2$.  Then $\conv(\Dc(K_n,\mathfrak{c})) = {\Qc}'_n + (1,\ldots, 1)$ (Example \ref{2...2}).  Let $c_1=\cdots =c_n=1$.  It follows that $\conv(\Dc(K_n,\mathfrak{c}))$ is Gorenstein if and only if either $n$ is even or $n=3$ (Example \ref{ham}).  
\end{proof}

\section{Complete bipartite graphs} \label{sec5}
Let $m\geq 1, n\geq 1$ be integers with $n+m \geq 3$ and $K_{m,n}$ the complete bipartite graph on the vertex set $\{x_1, \ldots, x_m\} \sqcup \{x_{m+1}, \ldots, x_{m+n}\}$.  Let $\mathfrak{c}=(c_1, \ldots, c_{m+n})\in (\ZZ_{>0})^{m+n}$.

\begin{Example}
\label{ABCDE}
Suppose that $c_1+\cdots +c_m=c_{m+1}+\cdots +c_{m+n}$.   One has $\Bc(K_{m,n},\mathfrak{c}) = \{x_1^{c_1}x_2^{c_2} \cdots x_{m+n}^{c_{m+n}} \}$ and $\rho_{(K_{m,n},\mathfrak{c})}(X) = \sum_{i \in X} c_i$ for $X \subset [m+n]$.  It follows that $X \subset [n]$ is $\rho_{(K_{K_{m,n}},\mathfrak{c})}$-closed and $\rho_{(K_{m,n},\mathfrak{c})}$-inseparable if and only if $|X| = 1$.  Hence, $\conv(\Dc(K_{m,n},\mathfrak{c}))$ is Gorenstein if and only if either $c_1=\cdots = c_{m+n} = 1$ or $c_1=\cdots = c_{m+n} = 2$ (Lemma \ref{criterion}).  In particular, if $\conv(\Dc(K_{m,n},\mathfrak{c}))$ is Gorenstein, then $m=n$.  As a result, we obtain the Gorenstein polytopes $\Qc_{2n} \subset \RR^{2n}$ and $\Qc'_{2n}+(1,\ldots, 1) \subset \RR^{2n}$. 
\end{Example}

\begin{Example}
\label{FGHIJ}
(a) Let $n = 2m - 1$ with $m \geq 2$ and fix a subset $A$ of $[m+n] \setminus [m]$, possibly $A = \emptyset$ or $A = [m+n] \setminus [m]$.  Let $\mathfrak{c}=(c_1,\ldots,c_{m+n})\in (\ZZ_{>0})^{m+n}$, where $c_i = 1$ if $i \in [m+n] \setminus A$ and where $c_i = m$ if $i \in A$.  Then $\Bc(K_{m,n},\mathfrak{c})$ consists of those monomials $x_{1}\cdots x_{m}u$, where $u$ is a  monomial in $x_{m+1}, \ldots, x_{m+n}$ of degree $m$ bounded by $(c_{m+1}, \ldots, c_{m+n})$.
If either $X \cap A \neq \emptyset$ or $X=[m+n]\setminus [m]$, then $\rho_{(K_{K_{m,n}},\mathfrak{c})}(X) = m$.  It follows that $\rho_{(K_{K_{m,n}},\mathfrak{c})}$-closed and $\rho_{(K_{m,n},\mathfrak{c})}$-inseparable subsets of $[m+n]$ are the singleton $\{i\}$ for $i \in [m+n] \setminus A$ together with $[m+n]\setminus [m]$.  Since $\rho_{(K_{K_{m,n}},\mathfrak{c})}([m+n]\setminus [m])=m = (n + 1)/2$, it follows from Lemma \ref{criterion} that $\conv(\Dc(K_{m,n},\mathfrak{c}))$ is Gorenstein.  More precisely, $\conv(\Dc(K_{m,n},\mathfrak{c}))$ is defined by the linear inequalities $0 \leq x_i$ for $i \in [m+n]$, $x_i \leq 1$ for $i \notin A$ together with $$x_{m+1}+\cdots+x_{m+n} \leq m.$$

(b) Let $n = 2m - 1$ with $m \geq 2$ and fix a subset $A$ of $[m+n] \setminus [m]$, possibly $A = \emptyset$ or $A = [m+n] \setminus [m]$.  Let $\mathfrak{c}=(c_1,\ldots,c_{m+n})\in (\ZZ_{>0})^{m+n}$, where $c_i = 2$ if $i \in [m+n] \setminus A$ and where $c_i = 2m$ if $i \in A$.  A similar argument as in (a) shows that $\conv(\Dc(K_{m,n},\mathfrak{c}))$ is Gorenstein.  More precisely, $\conv(\Dc(K_{m,n},\mathfrak{c}))$ is defined by the linear inequalities $0 \leq x_i$ for $i \in [m+n]$, $x_i \leq 2$ for $i \notin A$ together with $$x_{m+1}+\cdots+x_{m+n} \leq 2m.$$ 
\end{Example}

We now come to the classification of Gorenstein polytopes arising from complete
bipartite graphs.

\begin{Theorem}
    \label{completebipartite}
Let $m\geq 1, n\geq 1$ be integers with $n+m \geq 3$ and $K_{m,n}$ the complete bipartite graph on the vertex set $\{x_1, \ldots, x_m\} \sqcup \{x_{m+1}, \ldots, x_{m+n}\}$.  The Gorenstein polytopes of the form $\conv(\Dc(K_{m,n},\mathfrak{c}))$ are those of Examples \ref{ABCDE} and \ref{FGHIJ}   
\end{Theorem} 

\begin{proof}
If $c_1+\cdots +c_m=c_{m+1}+\cdots +c_{m+n}$, then  $\conv(\Dc(K_{m,n},\mathfrak{c}))$ is one of the polytopes presented in Example \ref{ABCDE}.  Suppose that $c_1+\cdots +c_m\neq c_{m+1}+\cdots +c_{m+n}$.  Let, say, $c_1+\cdots +c_m < c_{m+1}+\cdots +c_{m+n}$. Note that for a monomial $u\in S$, one has $u\in \Bc(K_{m,n},\mathfrak{c})$ if and only if $u$ can be written as $x_1^{c_1}\cdots x_m^{c_m}u_1$, where $u_1$ is a $(c_{m+1}, \ldots, c_{m+n})$-bounded monomial of degree $c_1+\cdots +c_m$ on variables $x_{m+1}, \ldots, x_{m+n}$.  For each $i=1, \ldots, m$, the singleton $\{i\}$ is a $\rho_{(K_{K_{m,n}},\mathfrak{c})}$-closed and $\rho_{(K_{m,n},\mathfrak{c})}$-inseparable subset of $[m+n]$ with $\rho_{(K_{K_{m,n}},\mathfrak{c})}(\{i\})=c_i$. It is clear that the set $\{m+1, \ldots, m+n\}$ is a $\rho_{(K_{K_{m,n}},\mathfrak{c})}$-closed subset of $[m+n]$ with $$\rho_{(K_{K_{m,n}},\mathfrak{c})}(\{m+1, \ldots, m+n\})=c_1+\cdots +c_m.$$We show that this set is $\rho_{(K_{m,n},\mathfrak{c})}$-inseparable.  Suppose that $A_1$ and $A_2$ are proper subsets of $\{m+1, \ldots, m+n\}$ with $A_1\cap A_2=\emptyset$ and $A_1\cup A_2=\{m+1, \ldots, m+n\}$. Then for $j=1,2$, one has$$\rho_{(K_{m,n},\mathfrak{c})}(A_j)=\min\big\{c_1+ \cdots + c_m, \sum_{k\in A_j}c_k\big\}.$$Since $c_1+\cdots +c_m < c_{m+1}+\cdots +c_{m+n}$, the above equality implies that$$\rho_{(K_{m,n},\mathfrak{c})}(A_1)+\rho_{(K_{m,n},\mathfrak{c})}(A_2)> c_1+\cdots + c_m= \rho_{(K_{m,n},\mathfrak{c})}(\{1, \ldots, m\}).$$Therefore, $\{m+1, \ldots, m+n\}$ is a $\rho_{(K_{m,n},\mathfrak{c})}$-inseparable subset of $[m+n]$.

Now, by Lemma \ref{criterion}, there is an integer $k\geq 1$ such that for any $\rho_{(K_{m,n},\mathfrak{c})}$-closed and $\rho_{(K_{m,n},\mathfrak{c})}$-inseparable subsets $X \subset [m+n]$,
\begin{eqnarray}
\label{formula}
\rho_{(K_{m,n},\mathfrak{c})}(X)=\frac{1}{\, k \,}(|X|+1).
\end{eqnarray}
For each integer $i\in [m+n]$, set $\rho_i=\rho_{(K_{m,n},\mathfrak{c})}(\{i\})$. In particular, $\rho_i=c_i$, for each $i\in [m]$. In the preceding paragraph, we showed that the singletons $\{1\}, \ldots, \{m\}$ are $\rho_{(K_{m,n},\mathfrak{c})}$-closed and $\rho_{(K_{m,n},\mathfrak{c})}$-inseparable. So, the above equality implies that either $k=2$ and $\rho_1=\cdots =\rho_m=1$, or $k=1$ and $\rho_1=\cdots =\rho_m=2$. Therefore, one has the following two cases.

\medskip

{\bf Case 1.} Assume that $k=2$ and $\rho_1=\cdots =\rho_m=1$. Since $\{m+1, \ldots, m+n\}$ is a $\rho_{(K_{m,n},\mathfrak{c})}$-closed and $\rho_{(K_{m,n},\mathfrak{c})}$-inseparable subset of $[m+n]$ with$$\rho_{(K_{m,n},\mathfrak{c})}(\{m+1, \ldots, m+n\})=c_1+\cdots +c_m=m,$$we deduce from equality (\ref{formula}) that $n=2m-1$. Since $\rho_1+\cdots +\rho_m=m$, one has $\rho_{\ell}\leq m$, for each $\ell\in [m+n]\setminus [m]$. If $2\leq \rho_{\ell}\leq m-1$ for some integer $\ell$ with $m+1\leq \ell\leq m+n$, then the singleton $\{\ell\}$ is a $\rho_{(K_{m,n},\mathfrak{c})}$-closed and $\rho_{(K_{m,n},\mathfrak{c})}$-inseparable subset of $[m+n]$ with $\rho_{(K_{m,n},\mathfrak{c})}(\{\ell\})=\rho_{\ell}\geq 2$. This contradicts  (\ref{formula}). Thus, for each $\ell\in [m+n]\setminus [m]$, one has either $\rho_{\ell}=1$ or $\rho_{\ell}=m$. This yields that $\conv(\Dc(K_{m,n},\mathfrak{c}))$ is one of the polytopes presented in Example \ref{FGHIJ} (a). 

\medskip

{\bf Case 2.} Assume that $k=1$ and $\rho_1=\cdots =\rho_m=2$. Recall that for each $i\in [m]$, one has $\rho_i=c_i$. Since $\{m+1, \ldots, m+n\}$ is a $\rho_{(K_{m,n},\mathfrak{c})}$-closed and $\rho_{(K_{m,n},\mathfrak{c})}$-inseparable subset of $[m+n]$ with$$\rho_{(K_{m,n},\mathfrak{c})}(\{m+1, \ldots, m+n\})=c_1+\cdots +c_m=2m,$$we deduce from equality (\ref{formula}) that $n=2m-1$. Since $\rho_1+\cdots +\rho_m=2m$, one has $\rho_{\ell}\leq 2m$, for each $\ell\in [m+n]\setminus [m]$. If $1\leq \rho_{\ell}\leq 2m-1$ for some integer $\ell$ with $m+1\leq \ell\leq m+n$, then the singleton $\{\ell\}$ is a $\rho_{(K_{m,n},\mathfrak{c})}$-closed and $\rho_{(K_{m,n},\mathfrak{c})}$-inseparable subset of $[m+n]$ with $\rho_{(K_{m,n},\mathfrak{c})}(\{\ell\})=\rho_{\ell}$. Hence, equality (\ref{formula}) implies that $\rho_{\ell}=2$. Consequently, for each $\ell\in [m+n]\setminus [m]$, one has either $\rho_{\ell}=2$ or $\rho_{\ell}=2m$.  As a result, $\conv(\Dc(K_{m,n},\mathfrak{c}))$ is one of the polytopes presented in Example \ref{FGHIJ} (b). 
\end{proof}

\section{Paths} \label{sec6}
Let $n\geq 3$ and $P_n$ be the path of length $n-1$ on $\{x_1, \ldots, x_n\}$ whose edges are $$\{x_1,x_2\},\{x_2,x_3\}, \ldots, \{x_{n-1},x_n\}.$$

\begin{Example}
    \label{pathAAAAA}
    Let $n \geq 4$ be an even integer. If $\mathfrak{c}=(1,\ldots, 1) \in (\ZZ_{>0})^n$, then one has $\conv(\Dc(P_n,\mathfrak{c})) = \Qc_n \subset \RR^n$ (Example \ref{perfectmatching}).  Furthermore, if $\mathfrak{c}=(2,\ldots, 2) \in (\ZZ_{>0})^n$, then one has $\conv(\Dc(P_n,\mathfrak{c})) = \Qc'_n + (1,\ldots, 1) \subset \RR^n$ (Example \ref{2...2}).  
\end{Example}

\begin{Example}
    \label{pathBBBBB}
Let $n=5$.
\begin{itemize}
    \item[(i)] 
Let $\mathfrak{c}=(1,1,1,1,1)$.  One has $$\Bc(P_5,\mathfrak{c})=\{x_1x_2x_3x_4, x_1x_2x_4x_5, x_2x_3x_4x_5\}.$$  The $\rho_{(P_5,\mathfrak{c})}$-closed and $\rho_{(P_5,\mathfrak{c})}$-inseparable subsets are $\{1\}, \ldots, \{5\}$ and $\{1,3,5\}$.  Since $\rho_{(P_5,\mathfrak{c})}(\{1,3,5\})=2$, it follows from Lemma \ref{criterion} that $\conv(\Dc(P_5,\mathfrak{c}))$ is Gorenstein.  In fact, $\conv(\Dc(P_5,\mathfrak{c}))$ is defined by the system of linear inequalities $0 \leq x_i \leq 1$ for $1 \leq i \leq 5$ together with $x_1+x_3+x_5 \leq 2$.
    \item[(ii)] 
Let $\mathfrak{c}=(1,1,2,1,1)$.  One has $$\Bc(P_5,\mathfrak{c})=\{x_1x_2x_3x_4, x_1x_2x_4x_5, x_2x_3^2x_4, x_2x_3x_4x_5\}.$$  The $\rho_{(P_5,\mathfrak{c})}$-closed and $\rho_{(P_5,\mathfrak{c})}$-inseparable subsets are $\{1\}, \{2\}, \{4\}, \{5\}$ and $\{1,3,5\}$.  One has $\rho_{(P_5,\mathfrak{c})}(\{1,3,5\})=2$.  It  follows from Lemma \ref{criterion} that $\conv(\Dc(P_5,\mathfrak{c}))$ is Gorenstein.  In fact, $\conv(\Dc(P_5,\mathfrak{c}))$ is defined by the system of linear inequalities $0 \leq x_i \leq 1$ for $i=1,2,4,5$, $0 \leq x_3$ together with $x_1+x_3+x_5 \leq 2$.
    \item[(iii)] 
Let $\mathfrak{c}=(2,2,2,2,2)$. One has 
\begin{eqnarray*}
  \Bc(P_5,\mathfrak{c})=\{x_1^2x_2^2x_3^2x_4^2,x_1^2x_2^2x_4^2x_5^2, x_2^2x_3^2x_4^2x_5^2,
  \, \, \, \, \, \, \, \, \, \, \, \, \, \, \, \, \, \, \, 
  \\
  x_1x_2^2x_3^2x_4^2x_5, x_1x_2^2x_3x_4^2x_5^2, x_1^2x_2^2x_3x_4^2x_5\}.  
\end{eqnarray*}
  The $\rho_{(P_5,\mathfrak{c})}$-closed and $\rho_{(P_5,\mathfrak{c})}$-inseparable subsets are $\{1\}, \ldots, \{5\}$ and $\{1,3,5\}$.  Since $\rho_{(P_5,\mathfrak{c})}(\{1,3,5\})=4$, it follows from Lemma \ref{criterion} that $\conv(\Dc(P_5,\mathfrak{c}))$ is Gorenstein.  In fact, $\conv(\Dc(P_5,\mathfrak{c}))$ is defined by the system of linear inequalities $0 \leq x_i \leq 2$ for $1 \leq i\leq 5$ together with $x_1+x_3+x_5 \leq 4$.
    \item[(iv)] 
Let $\mathfrak{c}=(2,2,4,2,2)$.  One has 
\begin{eqnarray*}
  \Bc(P_5,\mathfrak{c})=\{x_1^2x_2^2x_3^2x_4^2,x_1^2x_2^2x_4^2x_5^2, x_2^2x_3^2x_4^2x_5^2,
  \, \, \, \, \, \, \, \, \, \, \, \, \, \, \, \, \, \, \, 
  \\
  x_1x_2^2x_3^2x_4^2x_5, x_1x_2^2x_3x_4^2x_5^2, x_1^2x_2^2x_3x_4^2x_5
  \, \, \, \, \, \, \, \\
  x_2^2x_3^4x_3^2, x_1x_2^2x_3^3x_4^2, x_2^2x_3^3x_4^2x_5 
  \}.  
\end{eqnarray*}
The $\rho_{(P_5,\mathfrak{c})}$-closed and $\rho_{(P_5,\mathfrak{c})}$-inseparable subsets are $\{1\}, \{2\}, \{4\}, \{5\}$ and $\{1,3,5\}$.  One has $\rho_{(P_5,\mathfrak{c})}(\{1,3,5\})=4$.  It follows from Lemma \ref{criterion} that $\conv(\Dc(P_5,\mathfrak{c}))$ is Gorenstein.  In fact, $\conv(\Dc(P_5,\mathfrak{c}))$ is defined by the system of linear inequalities $0 \leq x_i \leq 2$ for $i=1,2,4,5$, $0 \leq x_3$ together with $x_1+x_3+x_5 \leq 4$.
\end{itemize}
\end{Example}

\begin{Lemma} \label{pathodd1}
Let $n\geq 7$ be an odd integer and $\mathfrak{c}=(1, \ldots, 1)\in (\ZZ_{>0})^n$.  Then  $\conv(\Dc(P_n,\mathfrak{c}))$ is not  Gorenstein.     
\end{Lemma}

\begin{proof}
One easily sees that the sets $\{1\}$ and $\{1,3,5, \ldots, n\}$ are $\rho_{(P_n,\mathfrak{c})}$-closed and $\rho_{(P_n,\mathfrak{c})}$-inseparable with $\rho_{(P_n,\mathfrak{c})}(\{1\})=1$ and $\rho_{(P_n,\mathfrak{c})}(\{1,3,5,\ldots, n\})=(n-1)/2$.  Hence,  $\conv(\Dc(P_n,\mathfrak{c}))$ is not Gorenstein (Lemma \ref{criterion}).     
\end{proof}

\begin{Lemma} \label{pathodd2}
Let $n\geq 7$ be an odd integer and $\mathfrak{c}=(2, \ldots, 2)\in (\ZZ_{>0})^n$.  Then  $\conv(\Dc(P_n,\mathfrak{c}))$ is not  Gorenstein.     
\end{Lemma}

\begin{proof}
One easily sees that the sets $\{1\}$ and $\{1,3,5, \ldots, n\}$ are $\rho_{(P_n,\mathfrak{c})}$-closed and $\rho_{(P_n,\mathfrak{c})}$-inseparable with $\rho_{(P_n,\mathfrak{c})}(\{1\})=2$ and $\rho_{(P_n,\mathfrak{c})}(\{1,3,5,\ldots, n\})=n-1$.  Hence,  $\conv(\Dc(P_n,\mathfrak{c}))$ is not Gorenstein (Lemma \ref{criterion}).     
\end{proof}
    
We now come to the classification of Gorenstein polytopes arising from paths.    

\begin{Theorem}
    \label{pathMAIN}
Let $P_n$ be the path of length $n-1$ with $n \geq 3$.  The Gorenstein polytopes of the form $\conv(\Dc(P_n,\mathfrak{c}))$ are those of Examples \ref{pathAAAAA} and \ref{pathBBBBB}   
\end{Theorem} 

\begin{proof}
Since $P_3=K_{1,2}$, it follows from Theorem \ref{completebipartite} that for any $\mathfrak{c}\in (\ZZ_{>0})^3$, the polytope $\conv(\Dc(P_3,\mathfrak{c}))$ is not Gorenstein. So, assume that $n\geq 4$. Let $\mathfrak{c}\in (\ZZ_{>0})^n$ and suppose that $\conv(\Dc(P_n,\mathfrak{c}))$ is Gorenstein. For every integer $i=1, \ldots, n$, set $\rho_i:=\rho_{(P_n,\mathfrak{c})}(\{i\})$. Note that for each $i\notin \{3, n-2\}$ and for each $j\neq i$, we have $N_{P_n}(x_j)\nsubseteq N_{P_n}(x_i)$. Thus, Lemma \ref{noinclusion1} shows that the singleton $\{i\}$ is $\rho_{(P_n,\mathfrak{c})}$-closed and $\rho_{(P_n,\mathfrak{c})}$-inseparable. It follows from Lemma \ref{criterion} that either $\rho_i=1$, for each $i\in [n]\setminus \{3, n-2\}$, or $\rho_i=2$, for each $i\in [n]\setminus \{3, n-2\}$. For each $\ell\in \{3, n-2\}$, let $A_{\ell}$ be a maximal subset of $[n]$ containing $\ell$ such that $\rho_{(P_n,\mathfrak{c})}(A_{\ell})=\rho_{\ell}$. Assume that $A_{\ell}', A_{\ell}''$ are nonempty disjoint subsets of $A_{\ell}$ with $A_{\ell}=A_{\ell}'\cup A_{\ell}''$. Without loss of generality, we may assume that $\ell\in A_{\ell}'$. Thus, $\rho_{(P_n,\mathfrak{c})}(A_{\ell}')=\rho_{\ell}=\rho_{(P_n,\mathfrak{c})}(A_{\ell})$. Consequently, $\rho_{(P_n,\mathfrak{c})}(A_{\ell}')+\rho_{(P_n,\mathfrak{c})}(A_{\ell}'') > \rho_{(P_n,\mathfrak{c})}(A_{\ell})$. This inequality shows that $A_{\ell}$ is $\rho_{(P_n,\mathfrak{c})}$-inseparable. We divide the rest of the proof into the following cases. 

\medskip

{\bf Case 1.} Suppose that $\rho_i=1$, for each $i\in [n]\setminus \{3, n-2\}$. Since for each $i\notin\{1,n\}$, we have ${\rm deg}_{P_n}(x_j)\geq 2$, it follows from Lemma \ref{rhomin} that $c_i=1$, for each $i\notin\{1,3,n-2,n\}$. 

First, assume that $n=5$. Then it folllows from the preceding paragraph that $c_2=c_4=1$. Since $x_1$ and $x_5$ are leaves of $P_5$ and their unique neighbors are $x_2$, $x_4$, respectively, we deduce that $\rho_1=\rho_5=1$. Moreover, it follows from $N_{P_5}(x_3)=\{x_2,x_4\}$ that $\rho_3\leq 2$.  As a result, $\conv(\Dc(P_5,\mathfrak{c}))$ is one of the polytopes presented in Example \ref{pathBBBBB} (i)-(ii).  

Now, suppose that $n\neq 5$. Thus, $n=4$ or $n\geq 6$. If $\{3\}$ and $\{n-2\}$ are $\rho_{(P_n,\mathfrak{c})}$-closed, then we conclude from Lemma \ref{criterion} and our assumption in this case that $\rho_3=\rho_{n-2}=1$. Hence,  $\Bc(P_n,\mathfrak{c})=\Bc(P_n,\mathfrak{c'})$, where $\mathfrak{c'}=(1,\ldots, 1) \in (\ZZ_{>0})^n$.  Lemma \ref{pathodd1} implies that $n$ is even. Consequently, $\conv(\Dc(P_n,\mathfrak{c}))$ is the polytope presented in Example \ref{pathAAAAA}. Now, suppose that there is an integer $\ell\in \{3, n-2\}$ such that $\{\ell\}$ is not $\rho_{(P_n,\mathfrak{c})}$-closed. Let $A_{\ell}$ be the set defined in the first paragraph of the proof. Hence, $|A_{\ell}|\geq 2$. Note that for each integer $j\in [n]$, with $N_{P_n}(x_j)\nsubseteq N_{P_n}(x_{\ell})$, we conclude from the proof of Lemma \ref{noinclusion1} that $\rho_{(P_n,\mathfrak{c})}(\{j,\ell\}) > \rho_{(P_n,\mathfrak{c})}(\{\ell\})$. In particular, $j\notin A_\ell$. This conclusion together with the structure of $P_n$ shows $A_{\ell}\setminus\{\ell\}\subseteq \{1,n\}$ and (since $n\neq 5$) equality never holds. Thus, $|A_{\ell}|=2$. It follows from the maximality of $A_{\ell}$ that it is  $\rho_{(P_n,\mathfrak{c})}$-closed. Also, recall from the first paragraph of the proof that $A_{\ell}$ is $\rho_{(P_n,\mathfrak{c})}$-inseparable. This contradicts Lemma \ref{criterion}, as $|A_{\ell}|+1=3$ is odd.

\medskip

{\bf Case 2.} Suppose that $\rho_i=2$, for each $i\in [n]\setminus \{3, n-2\}$. If $\{3\}$ and $\{n-2\}$ are $\rho_{(P_n,\mathfrak{c})}$-closed, then we conclude from Lemma \ref{criterion} and our assumption in this case that $\rho_3=\rho_{n-2}=2$. Hence, it follows from Lemma \ref{pathodd2} that either $n=5$ or $n$ is even. Thus $\conv(\Dc(P_n,\mathfrak{c}))$ is one of the polytopes in Examples \ref{pathAAAAA} and \ref{pathBBBBB} (iii).  

Now, suppose that there is an integer $\ell\in \{3, n-2\}$, say $\ell=3$, such that $\{\ell\}$ is not $\rho_{(P_n,\mathfrak{c})}$-closed. As defined in the first paragraph of the proof, let $A_3$ be the maximal subset of $[n]$ containing $3$ such that $\rho_{(P_n,\mathfrak{c})}(A_3)=\rho_3$. Hence, $|A_3|\geq 2$. By the same argument as in Case 1, we have $A_3=\{1,3\}$ if $n\neq 5$, and $A_3\subseteq \{1,3,5\}$ if $n=5$. In particular, $2\leq |A_3|\leq 3$. It follows from the maximality of $A_3$ that it is  $\rho_{(P_n,\mathfrak{c})}$-closed. Also, recall from the first paragraph of the proof that $A_3$ is $\rho_{(P_n,\mathfrak{c})}$-inseparable. First, suppose that $|A_3|=2$. We deduce from Lemma \ref{criterion} and our assumption in this case that $\rho_3=\rho_{(P_n,\mathfrak{c})}(A_3)=3$. 

\medskip

{\bf Claim.} $\rho_{n-2}\neq 1$.

\medskip

{\it Proof of the claim.} Assume that $\rho_{n-2}=1$. Note that $\{n-2\}$ is $\rho_{(P_n,\mathfrak{c})}$-inseparable. so, it cannot be $\rho_{(P_n,\mathfrak{c})}$-closed, as otherwise it contradicts Lemma \ref{criterion}. Since for each $j\in[n]\setminus \{1,n\}$, we have $N_{P_n}(x_j)\nsubseteq N_{P_n}(x_{n-2})$, we conclude from the proof of Lemma \ref{noinclusion1} that $\rho_{(P_n,\mathfrak{c})}(\{n-2, j\})> \rho_{(P_n,\mathfrak{c})}(\{n-2\})$. Moreover, the same argument shows that if $n\neq 5$, then $\rho_{(P_n,\mathfrak{c})}(\{1, n-2\})> \rho_{(P_n,\mathfrak{c})}(\{n-2\})$. We prove  that $\rho_{(P_n,\mathfrak{c})}(\{n-2, n\})> \rho_{(P_n,\mathfrak{c})}(\{n-2\})$ and if $n=5$, then $\rho_{(P_5,\mathfrak{c})}(\{1, 3\})> \rho_{(P_n,\mathfrak{c})}(\{3\})$. This yields that $\{n-2\}$ is $\rho_{(P_n,\mathfrak{c})}$-closed, a contradiction. Let $v\in \Bc(P_n,\mathfrak{c})$ be a monomial with ${\rm deg}_{x_{n-2}}(v)=\rho_{n-2}=1$. If $x_n$ does not divide $v$, then it follows from ${\rm deg}_{x_{n-2}}(v)=1$ that ${\rm deg}_{x_{n-1}}(v)\leq 1$. Since $c_{n-1}\geq \rho_{n-1}\geq 2$, we deduce that $(x_{n-1}x_n)v\in (I(P_n)^{\delta_{\mathfrak{c}}(I(P_n))+1})_{\mathfrak{c}}$, a contradiction. Thus, $x_n$ divides $v$ which implies that $\rho_{(P_n,\mathfrak{c})}(\{n-2, n\})\geq 2 > \rho_{(P_n,\mathfrak{c})}(\{n-2\})$. Similarly, if $n=5$, then $\rho_{(P_5,\mathfrak{c})}(\{1, 3\})> \rho_{(P_5,\mathfrak{c})}(\{3\})$. This completes the proof of the claim.

\medskip

Note that by our assumption in this case,  $c_2\geq 2$ and $c_4\geq2$ (the inequalities follow from the claim  if $n=4$ or $6$). We show that $1\notin A_3$. To prove this, it is enough to show that $\rho_{(P_n,\mathfrak{c})}(\{1, 3\})> 3=\rho_3=\rho_{(P_n,\mathfrak{c})}(\{3\})$. Let $u\in \Bc(P_n,\mathfrak{c})$ be a monomial with ${\rm deg}_{x_3}(u)=\rho_3=3$ and suppose that $u=e_1\cdots e_{\delta}$, where $\delta=\delta_{\mathfrak{c}}(I(P_n))$ and $e_1, \ldots, e_{\delta}$ are edges of $P_n$. If $x_1$ divides $u$, then it follows that$$\rho_{(P_n,\mathfrak{c})}(\{1, 3\}\geq {\rm deg}_{x_1}(u)+{\rm deg}_{x_3}(u)\geq 4.$$Suppose that $x_1$ does not divide $u$. If ${\rm deg}_{x_2}(u)<2\leq c_2$, then $(x_1x_2)u\in (I(P_n)^{\delta+1})_{\mathfrak{c}}$, a contradiction. Thus, ${\rm deg}_{x_2}(u)\geq 2$. In particular, in the representation of $u$ as $u=e_1\cdots e_{\delta}$, there is an edge, say, $e_1$ which is equal to $\{x_2, x_3\}$. If ${\rm deg}_{x_4}(u) < 2\leq c_4$, then$$(x_1x_4)u=(x_1x_2)(x_3x_4)e_2\cdots e_{\delta}\in (I(P_n)^{\delta+1})_{\mathfrak{c}},$$a contradiction. Therefore, ${\rm deg}_{x_4}(u)\geq 2$ Since $\rho_3=3$ and ${\rm deg}_{x_2}(u)\geq 2$, it follows that in the representation of $u$, there is an edge, say, $e_2$ which is equal to $\{x_4, x_5\}$. Thus$$\frac{ux_1}{x_5}=(x_1x_2)(x_3x_4)e_3\cdots e_{\delta}\in \Bc(P_n,\mathfrak{c}).$$Hence
\begin{align*}
\rho_{(P_n,\mathfrak{c})}(\{1, 3\})& \geq {\rm deg}_{x_1}(vx_1/x_5)+{\rm deg}_{x_3}(vx_1/x_5)\\ &>{\rm deg}_{x_3}(vx_1/x_5)={\rm deg}_{x_3}(u)=3.
\end{align*}
So, $1\notin A_3$. Similarly, if $n=5$, one can show that $5\notin A_3$. This is a contradiction, as $A_3=\{1,3\}$ if $n\neq 5$, and $A_3\subseteq \{1,3,5\}$ if $n=5$.

Suppose that $|A_3|=3$. Therefore, $n=5$ and $A_3=\{1,3,5\}$.  One has $\rho_3=\rho_{(P_n,\mathfrak{c})}(A)=4$ (Lemma \ref{criterion}). Since $\rho_1=\rho_2=\rho_4=\rho_5=2$, it follows that $\conv(\Dc(P_n,\mathfrak{c}))$ is the polytopes presented in Example \ref{pathBBBBB} (iv).  
\end{proof}

\section{Regular bipartite graphs} \label{sec7}
We now turn to the discussion of finding Gorenstein polytopes $\conv(\Dc(G,\mathfrak{c}))$ arising from connected regular bipartite graphs. A finite graph $G$ on $\{x_1, \ldots, x_n\}$ is called {\em $k$-regular} if ${\rm deg}_G(x_i)=k$ for all $1 \leq i \leq n$.  
\begin{Lemma}
\label{regular}
Let $G$ be a connected $k$-regular (not necessarily bipartite) graph on $n\geq 3$ vertices and $\mathfrak{c}=(c_1, \ldots, c_n)\in (\ZZ_{>0})^n$. If $\conv(\Dc(G,\mathfrak{c}))$ is Gorenstein, then either $c_1=\cdots =c_n=1$ or $c_1=\cdots =c_n=2$.
\end{Lemma} 

\begin{proof}
Let $V(G)=\{x_1, \ldots, x_n\}$ and set $\rho_i:=\rho_{(G,\mathfrak{c})}(\{i\})$, for each $i=1, \ldots, n$. It follows from $n\geq 3$ that $G\neq K_2$ and so, $k\geq 2$. We consider the following two cases.

\medskip

{\bf Case 1.} Suppose that every singleton $\{i\}$ is $\rho_{(G,\mathfrak{c})}$-closed. Since every singleton is $\rho_{(G,\mathfrak{c})}$-inseparable, we conclude from Lemma \ref{criterion} that either, $\rho_1=\cdots = \rho_n = 1$ or $\rho_1=\cdots = \rho_n = 2$. We show that $\rho_i=c_i$, for each $i\in [n]$, and this completes the proof in this case. If $c_i\leq \sum_{x_t\in N_G(x_i)}c_t$, then the assertion follows from Lemma \ref{rhomin}. So, suppose that $c_i> \sum_{x_t\in N_G(x_i)}c_t$. Again using Lemma \ref{rhomin}, we deduce that $\rho_t=c_t$, for each integer $t$ with $x_t\in N_G(x_i)$. Moreover, $\rho_i=\sum_{x_t\in N_G(x_i)}c_t$. Since $k\geq 2$, there are two distinct vertices $x_{t_1},x_{t_2}\in N_G(x_i)$. It follows that$$\rho_i=\sum_{x_t\in N_G(x_i)}c_t\geq c_{t_1}+c_{t_2}=\rho_{t_1}+\rho_{t_2}.$$This is a contradiction, as $\rho_1=\cdots =\rho_n$.

\medskip

{\bf Case 2.} Suppose that there is $i\in [n]$ for which $\{i\}$ is not $\rho_{(G,\mathfrak{c})}$-closed. Then there is a maximal subset $A\subset [n]$ containing $i$ with $\rho_{(G,\mathfrak{c})}(A)=\rho_i$.  In particular, $A$ is a $\rho_{(G,\mathfrak{c})}$-closed subset of $[n]$ and $|A|\geq 2$. Let $j\in A$ with $j\neq i$. Also, let $u\in \Bc(G,\mathfrak{c})$ be a monomial with ${\rm deg}_{x_i}(u)=\rho_i$. Since $\rho_{(G,\mathfrak{c})}(A)=\rho_i$, we deduce that $\rho_{(G,\mathfrak{c})}(\{i,j\})=\rho_i$. The same argument as in the proof of Lemma \ref{noinclusion1} guarantees that $N_G(x_j)\subseteq N_G(x_i)$. Since $G$ is a $k$-regular graph, it follows that for any $j\in A$, the equality $N_G(x_j)=N_G(x_i)$ holds. In particular, $A$ is an independent set of $G$. Moreover, as the degree of every vertex in $N_G(x_i)$ is $k$, one has $|A|\leq k$. If $|A|=k$, then connectedness of $G$ says that $G=K_{k,k}$. Hence, Theorem \ref{completebipartite} implies that either $\rho_1=\cdots =\rho_n=1$ or $\rho_1=\cdots =\rho_n=2$. Then the same argument as in Case 1 yields that $c_i=\rho_i$, for each $i\in [n]$. So, suppose that $|A| < k$.

\medskip

{\bf Claim.} $\rho_i\geq k$.

\medskip

{\it Proof of the claim.} Set $\delta:=\delta_{\mathfrak{c}}(I(G))$. Since $u\in \Bc(G,\mathfrak{c})$, we can write $u=e_1\cdots e_{\delta}$, where $e_1, \ldots, e_{\delta}$ are edges of $G$. Recall that in the the preceding paragraph, we proved that $N_G(x_j)=N_G(x_i)$, for every $j\in A$. Moreover, for each $j\in A$ with $j\neq i$, we have $\rho_{(G,\mathfrak{c})}(\{i,j\})=\rho_i$. Thus, $x_j$ does not divide $u$. Consider a vertex $x_r\in N_G(x_j)=N_G(x_i)$. If ${\rm deg}_{x_r}(u)< c_r$, then $(x_jx_r)u\in (I(G)^{\delta+1})_{\mathfrak{c}}$ which is a contradiction. This contradiction shows that for any vertex $x_r\in N_G(x_i)$, we have ${\rm deg}_{x_r}(u)=c_r$. Fix a vertex $x_r\in N_G(x_i)$. It follows that $x_r$ divides $u$. If in the representation of $u$ as $u=e_1\cdots e_{\delta}$, there is an edge, say $e_1$, which is incident to $x_r$ but not to $x_i$, then $e_1=\{x_r, x_{r'}\}$ for some vertex $x_{r'}\in V(G)\setminus \{x_i\}$. Consequently, $$\frac{ux_j}{x_{r'}}=(x_jx_r)e_2\cdots e_{\delta}\in \Bc(G,\mathfrak{c}).$$Thus,
\begin{align*}
\rho_{(G,\mathfrak{c})}(\{i,j\})& \geq {\rm deg}_{x_i}(ux_j/x_{r'})+{\rm deg}_{x_j}(ux_j/x_{r'})>{\rm deg}_{x_i}(ux_j/x_{r'})\\ &={\rm deg}_{x_i}(u)=\rho_i,
\end{align*}
which is a contradiction. So, for any edge $\ell=1, \ldots, \delta$, if $x_r\in e_{\ell}$, then $e_{\ell}=\{x_i,x_r\}$. Therefore,
\begin{eqnarray}
\label{formula2}
\rho_i={\rm deg}_{x_i}(u)=\sum_{x_r\in N_G(x_i)}{\rm deg}_{x_r}(u)=\sum_{x_r\in N_G(x_i)}c_r.
\end{eqnarray}
Since ${\rm deg}_G(x_i)=k$, we conclude from the above equalities that $\rho_i\geq k$. This completes the proof of the claim. 

\medskip

Next, we show that $A$ is a $\rho_{(G,\mathfrak{c})}$-inseparable subset of $[n]$. Indeed suppose that $A_1$ and $A_2$ are disjoint subsets of $A$ with $A_1\cup A_2=A$.  We may assume that $i\in A_1$. Therefore, $\rho_{(G,\mathfrak{c})}(A_1)=\rho_i$. Consequently,$$\rho_{(G,\mathfrak{c})}(A_1)+\rho_{(G,\mathfrak{c})}(A_2)> \rho_i=\rho_{(G,\mathfrak{c})}(A).$$Therefore, $A$ is a $\rho_{(G,\mathfrak{c})}$-inseparable subset of $[n]$. Since $\conv(\Dc(G,\mathfrak{c}))$ is  Gorenstein, we conclude from Lemma \ref{criterion} and the inequality $|A|<k$ that $\rho_{(G,\mathfrak{c})}(A)\leq k$. Since $\rho_{(G,\mathfrak{c})}(A)=\rho_i$, it follows from that claim that $\rho_{(G,\mathfrak{c})}(A)=k$.  As $|A|<k$ and $\rho_{(G,\mathfrak{c})}(A)=k$, it follows from Lemma \ref{criterion} that $|A|=k-1$. Moreover, (\ref{formula2}) implies that $c_r=1$, for each integer $r$ with $x_r\in N_G(x_i)$. It follows that the singleton $\{r\}$ is a $\rho_{(G,\mathfrak{c})}$-closed subset of $[n]$. Obviously, it is $\rho_{(G,\mathfrak{c})}$-inseparable too. This contradicts Lemma \ref{criterion}, as $A$ is another $\rho_{(G,\mathfrak{c})}$-closed and $\rho_{(G,\mathfrak{c})}$-inseparable subset of $[n]$ with $|A|=k-1$ and $\rho_{(G,\mathfrak{c})}(A)=k$.
\end{proof}

We are now ready to characterize Gorenstein polytopes arising from regular bipartite graphs.

\begin{Theorem}
    \label{regularTH}  
The Gorenstein polytopes of the form $\conv(\Dc(G,\mathfrak{c}))$, where $G$ is a connected regular bipartite graph on $n\geq 3$ vertices and where $\mathfrak{c}\in (\ZZ_{>0})^n$, are exactly ${\Qc}'_n + (1,\ldots, 1)$ and $\Qc_n$.  
\end{Theorem}

\begin{proof}  
Recall that a regular bipartite graph has a perfect matching.  Set $\mathfrak{c_1}:=(1,1, \ldots, 1)\in (\ZZ_{>0})^n$ and $\mathfrak{c_2}:=(2,2, \ldots, 2)\in (\ZZ_{>0})^n$. Since $\conv(\Dc(G,\mathfrak{c}))$ is Gorenstein, it follows from Lemma \ref{regular} that either $\mathfrak{c}=\mathfrak{c_1}$ or $\mathfrak{c}=\mathfrak{c_2}$. The existence of perfect matching guaranrtees that $\Bc(G,\mathfrak{c_1}) = \{x_1x_2\cdots x_n\}$ and $\Bc(G,\mathfrak{c_2}) = \{x_1^2x_2^2\cdots x_n^2\}$. Hence, $\conv(\Dc(G,\mathfrak{c_1})) = \Qc_n$ and $\conv(\Dc(G,\mathfrak{c_2})) = \Qc'_n+(1,\ldots,1)$.   
\end{proof}

\begin{Example}
    Let $G$ be a connected regular non-bipartite graph on $n$ vertices and $\mathfrak{c}=(1,1, \ldots, 1)\in (\ZZ_{>0})^n$.  Then $\conv(\Dc(G,\mathfrak{c}))$ might not be Gorenstein, e.g., $G=C_5$ (Example \ref{cycle}).  Furthermore, even if $\conv(\Dc(G,\mathfrak{c}))$ is Gorenstein, $\conv(\Dc(G,\mathfrak{c}))$ might not be equal to $\Qc_n$, e.g., $G=C_3$ (Example \ref{ham}). 
\end{Example}

However, for $\mathfrak{c}=(2,2, \ldots, 2)\in (\ZZ_{>0})^n$, we have the following theorem.

\begin{Theorem} 
    \label{regnonbip}
Let $G$ be a (not necessarily bipartite) regular graph on $n$ vertices. Then for the vector $\mathfrak{c}=(2, 2, \ldots, 2)\in (\ZZ_{>0})^n$, the lattice polytope $\conv(\Dc(G,\mathfrak{c}))$ is $\Qc'_n+(1,\ldots,1)$. In particular, $\conv(\Dc(G,\mathfrak{c}))$ is Gorenstein.
\end{Theorem}

\begin{proof}
Assume that $G$ is a $k$-regular graph on vertex set $V(G)=\{x_1, \ldots, x_n\}$. We claim that $\Bc(G,\mathfrak{c}) = \{x_1^2x_2^2\cdots x_n^2\}$. To prove the claim it is enough to prove that $x_1^2x_2^2\cdots x_n^2\in \Bc(G,\mathfrak{c})$. If $k$ is even, then by Petersen's 2-factor theorem \cite[Page 166]{bcl}, the graph $G$ has a spanning subgraph $H$ which is disjoint union of cycles. Thus,$$x_1^2x_2^2\cdots x_n^2=\prod_{\{x_i,x_j\}\in E(H)}(x_ix_j)\in \Bc(G,\mathfrak{c}).$$If $k$ is odd, then it follows from \cite[Theorem 1]{k} that $G$ has a spanning subgraph $H$ such that every connected component of $H$ is either an edge or a cycle. Assume that $H_1, \ldots, H_s$ are those connected components of $H$ which are an edge and let $H_{s+1}, \ldots, H_t$ be the connected components of $H$ which are cycles. Then$$x_1^2x_2^2\cdots x_n^2=\Big(\prod_{\ell=1}^s\prod_{\{x_i,x_j\}\in E(H_{\ell})}(x_ix_j)^2\Big)\Big(\prod_{\ell=s+1}^t\prod_{\{x_i,x_j\}\in E(H_{\ell})}(x_ix_j)\Big)\in \Bc(G,\mathfrak{c}).$$Thus, $\Bc(G,\mathfrak{c}) = \{x_1^2x_2^2\cdots x_n^2\}$. Hence, $\conv(\Dc(G,\mathfrak{c})) = \Qc'_n+(1,\ldots,1)$.   
\end{proof}

\section{Whiskered graphs} \label{sec8}
Recall that every finite graph to be discussed in the present paper has no isolated vertices.  Let $G$ be a finite graph on $\{x_1, \ldots, x_n\}$.  The {\em whiskered graph} of $G$ is the finite graph $W(G)$ on $\{x_1, \ldots, x_{2n}\}$ obtained from $G$ by adding the edges $\{x_i, x_{n+i}\}$ for $1\leq i \leq n$.    

\begin{Lemma}
\label{whiskerLEMMA}
Let $G$ be a finite graph on $n$ vertices $x_1, \ldots, x_n$ and $\mathfrak{c} \in (\ZZ_{>0})^{2n}$.  Then  $\conv(\Dc(W(G),\mathfrak{c}))$ is Gorenstein if and only if one of the following conditions holds:
\begin{itemize}
 \item[(i)] 
$c_1=\cdots=c_n=1$ and $c_{n+i}\geq 1$ for each $i=1, \ldots, n$;
\item[(ii)]
$c_1=\cdots =c_n=2$ and $c_{n+i}\geq 2$ for each $i=1, \ldots, n$.
\end{itemize}
\end{Lemma}

\begin{proof}
Suppose that $\conv(\Dc(W(G),\mathfrak{c}))$ is Gorenstein.  Set $\delta:=\delta_{\mathfrak{c}}(I(W(G)))$.  In addition, for each $i=1, \ldots, 2n$, set $\rho_i:=\rho_{(W(G),\mathfrak{c})}(\{i\})$. We consider the following cases.

\medskip

{\bf Case 1.} Suppose that for each $i=1, \ldots, n$, the singleton $\{i\}$ is $\rho_{(W(G),\mathfrak{c})}$-closed. Obviously, every singleton is $\rho_{(W(G),\mathfrak{c})}$-inseparable.  We conclude from Lemma \ref{criterion} that either $\rho_1=\cdots = \rho_n = 1$ or $\rho_1=\cdots = \rho_n = 2$. In the first case, it follows from Lemma \ref{rhomin} that $c_1=\cdots =c_n=1$ (note that ${\rm deg}_{W(G)}(x_i)\geq 2$, for each $i=1, \ldots, n$). So, condition (i) holds.  Assume that $\rho_1=\cdots = \rho_n = 2$. It follows from these equalities that $c_i\geq 2$, for each $i=1, \ldots, n$ and again using Lemma \ref{rhomin}, we deduce that $c_1=\cdots =c_n=2$. We show that $c_{n+i}\geq 2$, for each $i=1, \ldots, n$. By contradiction, suppose that $c_{n+i}=1$, for some integer $i$ with $1\leq i\leq n$. We know from Lemma \ref{noinclusion1} that $\{n+i\}$ is $\rho_{(W(G),\mathfrak{c})}$-closed (and is $\rho_{(W(G),\mathfrak{c})}$-inseparable).  On the other hand, by our assumption, $\{i\}$ is $\rho_{(W(G),\mathfrak{c})}$-closed and $\rho_{(W(G),\mathfrak{c})}$-inseparable. Moreover, $\rho_i\neq \rho_{n+i}$. This contradicts Lemma \ref{criterion}.

\medskip

{\bf Case 2.} Suppose that there is an integer $1\leq i\leq n$ for which the singleton $\{i\}$ is not $\rho_{(W(G),\mathfrak{c})}$-closed. We may choose $i$ such that $\rho_i\leq \rho_t$ for each $t\in [n]$ with the property that the singleton $\{t\}$ is not $\rho_{(W(G),\mathfrak{c})}$-closed. Let $A$ be a maximal subset of $[2n]$ with $i\in A$ and $\rho_{(W(G),\mathfrak{c})}(A)=\rho_i$. In particular, $|A|\geq 2$.

\medskip

{\bf Claim 1.} There is a nonempty subset $B$ of $\{n+1, \ldots, 2n\}$ with $n+i\notin B$ for which $A=B\cup\{i\}$.

\medskip

{\it Proof of Claim 1.} Let $v\in \Bc(W(G),\mathfrak{c})$ be a monomial with ${\rm deg}_{x_i}(v)=\rho_i$. Then $v$ can be written as $v=f_1\cdots f_{\delta}$, where $f_1, \ldots, f_{\delta}$ are edges of $W(G)$. 

We first show that $n+i\notin A$. Indeed, if $x_{n+i}$ divides $v$, then$$\rho_{(W(G),\mathfrak{c})}(\{i,n+i\}) \geq {\rm deg}_{x_i}(v)+{\rm deg}_{x_{n+i}}(v)>{\rm deg}_{x_i}(v)=\rho_i.$$So, in this case, $n+i\notin A$. Therefore, assume that $x_{n+i}$ does not divide $v$. Since ${\rm deg}_{x_i}(v) \geq 1$, in the representation of $v$ as $v=f_1\cdots f_{\delta}$, there is an edge, say $f_1$ which is incident to $x_i$ but not to $x_{n+i}$. In other words, $f_1=\{x_i, x_{i'}\}$ for a vertex $x_{i'}\in V(W(G))\setminus \{x_{n+i}\}$. Consequently,$$\frac{vx_{n+i}}{x_{i'}}=(x_ix_{n+i})f_2\cdots f_{\delta}\in \Bc(W(G),\mathfrak{c}).$$Thus,
\begin{align*}
\rho_{(W(G),\mathfrak{c})}(\{i, n+i\})& \geq {\rm deg}_{x_i}(vx_{n+i}/x_{i'})+{\rm deg}_{x_{n+i}}(vx_{n+i}/x_{i'})\\ &>{\rm deg}_{x_i}(vx_{n+i}/x_{i'})={\rm deg}_{x_i}(v)=\rho_i.
\end{align*}
Hence, $n+i\notin A$.

Next, we show that for each $j\in [n]$ with $j\neq i$, we have $j\notin A$. Indeed, if $x_j$ divides $v$ a similar argument as above shows that $j\notin A$. If $x_j$ does not divide $v$, then $x_{n+j}$ does not divide $v$ and therefore, $(x_jx_{n+j})u\in (I(W(G))^{\delta+1})_{\mathfrak{c}}$ which is a contradiction. Consequently, $j\notin A$.

It follows from the preceding two paragraphs that $A=B\cup\{i\}$, for a subset $B$ of $\{n+1, \ldots, 2n\}$ with $n+i\notin B$. On the other hand, it follows from $|A|\geq 2$ that $B\neq \emptyset$. This completes the proof of Claim 1.

\medskip

{\bf Claim 2.} One has $\rho_{(W(G),\mathfrak{c})}(A)\geq c_{n+i}+\sum_{n+k\in B}c_k$.

\medskip

{\it Proof of Claim 2.} Let $v$ be the monomial defined in the proof of Claim 1. Assume that $n+k\in B\subset A$. Since $\rho_{(W(G),\mathfrak{c})}(\{i, n+k\})=\rho_i$, it follows that $v$ is not divisible by $x_{n+k}$. If ${\rm deg}_{x_k}(v)< c_k$, then  $(x_kx_{n+k})v\in (I(W(G))^{\delta+1})_{\mathfrak{c}}$ which is a contradiction. Thus, ${\rm deg}_{x_k}(v)= c_k$, for each integer $k$ with $n+k\in B$. Assume that in the representation of $v$ as $v=f_1\cdots f_{\delta}$, there is an edge, say $f_{\delta}$ which is incident to $x_k$ but not to $x_i$. Then $f_{\delta}=\{x_k,x_{k'}\}$, for some vertex $x_{k'}\neq x_i$. This yields that$$\frac{vx_{n+k}}{x_{k'}}=f_1\cdots f_{\delta-1}(x_kx_{n+k})\in \Bc(W(G),\mathfrak{c}).$$Thus,
\begin{align*}
\rho_{(W(G),\mathfrak{c})}(\{i, n+k\})& \geq {\rm deg}_{x_i}(vx_{n+k}/x_{k'})+{\rm deg}_{x_{n+k}}(vx_{n+k}/x_{k'})\\ &>{\rm deg}_{x_i}(vx_{n+k}/x_{k'})={\rm deg}_{x_i}(v)=\rho_i,
\end{align*}
which is a contradiction as $\rho_{(W(G),\mathfrak{c})}(\{i, n+k\})=\rho_i$. This contradiction shows that in the representation of $v$ as $v=f_1\cdots f_{\delta}$, if an edge $f_{\ell}$ is incident to $x_k$, it is incident to $x_i$ too. Moreover, if ${\rm deg}_{x_{n+i}}(v)<c_{n+i}$, then for an integer $k$ with $n+k\in B$, $$(x_{n+i}x_{n+k})v=(x_ix_{n+i})(x_kx_{n+k})v/(x_ix_k)\in (I(W(G))^{\delta+1})_{\mathfrak{c}},$$which is a contradiction. Hence, ${\rm deg}_{x_{n+i}}(v)=c_{n+i}$. Since $x_i$ is the unique neighbor of $x_{n+i}$ in $W(G)$, we deduce that$$\rho_{(W(G),\mathfrak{c})}(A)=\rho_i={\rm deg}_{x_i}(v)\geq {\rm deg}_{x_{n+i}}(v)+\sum_{n+k\in B}{\rm deg}_{x_k}(v)=c_{n+i}+\sum_{n+k\in B}c_k.$$This proves Claim 2.

\medskip

We show that $A$ is $\rho_{(W(G),\mathfrak{c})}$-inseparable. Indeed, assume that $A_1$ and $A_2$ are proper disjoint subsets of $A$ with $A_1\cup A_2=A$. We may assume that $i\in A_1$. Then $\rho_{(W(G),\mathfrak{c})}(A_1)=\rho_i$. Hence, $\rho_{(W(G),\mathfrak{c})}(A_1)+\rho_{(W(G),\mathfrak{c})}(A_2)> \rho_{(W(G),\mathfrak{c})}(A)$. Thus, $A$ is a $\rho_{(W(G),\mathfrak{c})}$-inseparable subset of $[2n]$. Moreover, since $A$ is a maximal subset of $[2n]$ with $\rho_{(W(G),\mathfrak{c})}(A)=\rho_i$, we conclude that $A$ is $\rho_{(W(G),\mathfrak{c})}$-closed. It follows from Lemma \ref{noinclusion1} that the singleton $\{n+i\}$ is a $\rho_{(W(G),\mathfrak{c})}$-closed and $\rho_{(W(G),\mathfrak{c})}$-inseparable subset of $[2n]$. Therefore, by Lemma \ref{criterion}, one has either $\rho_{n+i}=1$ or $\rho_{n+i}=2$. However, $\rho_{n+i}=1$ is not possible, as $A$ is $\rho_{(W(G),\mathfrak{c})}$-closed and $\rho_{(W(G),\mathfrak{c})}$-inseparable with $|A|\geq 2$ and $\rho_{(W(G),\mathfrak{c})}(A)\geq |A|$ (Claim 2). So, suppose that $\rho_{n+i}=2$. It follows from Lemma \ref{criterion} that $\rho_{(W(G),\mathfrak{c})}(A)=|A|+1$. Since $c_{n+i}\geq \rho_{n+i}=2$, we deduce from Claim 2 that $c_k=1$, for each integer $k$ with $n+k\in B$. On the other hand, at the beginning of Case 2, we assumed that $\rho_i\leq \rho_t$ for each $t\in [n]$ such that the singleton $\{t\}$ is not $\rho_{(W(G),\mathfrak{c})}$-closed. Since $\rho_i=\rho_{(W(G),\mathfrak{c})}(A)=|A|+1\geq 3$, it follows that the singleton $\{k\}$ is $\rho_{(W(G),\mathfrak{c})}$-closed, for each integer $k$ with $n+k\in B$. Obviously, it is $\rho_{(W(G),\mathfrak{c})}$-inseparable too. This contradicts Lemma \ref{criterion}, as $\{n+i\}$ is $\rho_{(W(G),\mathfrak{c})}$-closed and $\rho_{(W(G),\mathfrak{c})}$-inseparable with $\rho_{n+i}=2$.
\end{proof} 

We are now ready to prove the main result of this section.

\begin{Theorem}
    \label{whiskerTHM}
The Gorenstein polytopes of the form $\conv(\Dc(W(G),\mathfrak{c}))$, where $W(G)$ is the whiskered graph of a finite graph on $n$ vertices and where $\mathfrak{c}\in (\ZZ_{>0})^{2n}$, are exactly $\Qc_{2n}$ and ${\Qc}'_{2n} + (1,\ldots, 1)$. 
\end{Theorem}

\begin{proof}
Every whiskered graph has a perfect matching.  Hence, by virtue of Lemma \ref{whiskerLEMMA}, the proof of Theorem \ref{regularTH} remains valid without modification.    
\end{proof}

\section{Cohen--Macaulay Cameron--Walker graphs} \label{sec9}
Finally, we discuss Gorentein polytopes arising from Cohen--Macaulay Cameron--Walker graphs. Let $r \geq 1$ and $s \geq 1$ be integers and $H$ a connected bipartite graph on the vertex set $\{x_1\ldots, x_r\}\sqcup \{x_{2r+1}, \ldots, x_{2r+s}\}$.  We then define $H_s^r$ to be the finite graph on 
$\{x_1, \ldots x_{2r}, x_{2r+1}, \ldots, x_{2r+3s}\}$
for which 
\begin{itemize}
\item[(i)] the induced subgraph of $H_s^r$ on $\{x_1\ldots, x_r\}\sqcup \{x_{2r+1}, \ldots, x_{2r+s}\}$ is $H$, and
\item[(ii)] for each $i$ with $1\leq i\leq r$, there is exactly one pendant edge $\{x_i, x_{r+i}\}$ attached to $x_i$, and
\item [(iii)] for each $i$ with $1\leq i\leq s$, there is exactly one pendant triangle with vertices  $x_{2r+i}, x_{2r+s+i}, x_{2r+2s+i}$ attached to $x_{2r+i}$.
\end{itemize}
Recall from \cite[Theorem 1.3]{CW} that every Cohen--Macaulay Cameron--Walker graph is of the form $H_s^r$.

\begin{Lemma} \label{cmcw1}
Let $G=H_s^r$ be a Cohen--Macaulay Cameron--Walker graph on $n=2r+3s$ vertices and  $\mathfrak{c}=(1, \ldots, 1)\in (\ZZ_{>0})^n$. Then $\conv(\Dc(G),\mathfrak{c}))$ is not Gorenstein.
\end{Lemma}

\begin{proof}
We first show that $[n]$ is $\rho_{(G,\mathfrak{c})}$-inseparable. Indeed, let $A_1, A_2$ be proper disjoint subsets of $[n]$ with $A_1\cup A_2=[n]$. For $k=1,2$, set$$B_k:=A_k\cap\{r+1, \ldots, 2r\}, \ \ \ \ \ \ \ \ \ \ {\rm and} \ \ \ \ \ \ \ \ \ \ C_k:=A_k\cap\{2r+1, \ldots, 2r+s\}.$$Also, set $B_k':=\{i-r\mid i\in B_k\}$. Note that ${\rm match}(G)=r+s$ and$$w:=\frac{x_1x_2\cdots x_{2r+3s}}{x_{2r+1}\cdots x_{2r+s}}=\prod_{i=1}^r(x_ix_{i+r})\prod_{j=1}^s(x_{2r+s+j}x_{2r+2s+j})\in \Bc(G,\mathfrak{c}).$$This shows that $\rho_{(G,\mathfrak{c})}(A_k)\geq |A_k|-|C_k|$, for $k=1,2$. Notice that $B_1'\sqcup B_2'=\{1, \ldots,r\}$ and $C_1\sqcup C_2=\{2r+1, \ldots, 2r+s\}$. Since $G$ is a connected graph, either a vertex in $C_1$ is adjacent to a vertex in $B_2'$, or a vertex in $C_2$ is adjacent to a vertex in $B_1'$. Without loss of generality, we may assume that a vertex in $C_1$ is adjacent to a vertex in $B_2'$. In other words, there is a vertex $x_p\in C_1$ and a vertex $x_q\in B_2'$ such that $\{x_p,x_q\}$ is an edge of $G$. This yields that $$\frac{x_pw}{x_{q+r}}=\frac{(x_px_q)w}{(x_qx_{q+r})}\in \Bc(G,\mathfrak{c})$$which implies that$$\rho_{(G,\mathfrak{c})}(A_1)\geq \sum_{\ell \in A_1}{\rm deg}_{x_{\ell}}(x_pw/x_{q+r})=|A_1|-|C_1|+1.$$Consequently,
\begin{align*}  
& \rho_{(G,\mathfrak{c})}(A_1)+\rho_{(G,\mathfrak{c})}(A_2)\geq (|A_1|-|C_1|+1)+(|A_2|-|C_2|)\\ & =n-s+1=2r+3s-s+1 =2r+2s+1\\ & =\rho_{(G,\mathfrak{c})}([n])+1,
\end{align*}
where the last equality follows from ${\rm match}(G)=r+s$. Thus, $[n]$ is $\rho_{(G,\mathfrak{c})}$-inseparable. It is obvious that $[n]$ is $\rho_{(G,\mathfrak{c})}$-closed too. On the other hand, by Lemma \ref{noinclusion1}, the singleton $\{1\}$ is $\rho_{(G,\mathfrak{c})}$-closed and $\rho_{(G,\mathfrak{c})}$-inseparable. Therefore, Lemma \ref{criterion} says that the lattice polytope $\conv(\Dc(W(G),\mathfrak{c}))$ is not Gorenstein.
\end{proof}

\begin{Lemma}
     \label{cmcw}
Let $G=H_s^r$ be a Cohen--Macaulay Cameron--Walker graph on $2r+3s$ vertices and $\mathfrak{c}=(c_1, \ldots, c_{2r+3s})\in (\ZZ_{>0})^{2r+3s}$. Then  $\conv(\Dc(G),\mathfrak{c}))$ is Gorenstein if and only if the following conditions hold:
\begin{itemize}
\item [(i)] $c_i=2$ for each $i\in [2r+3s]\setminus \{r+1, \ldots, 2r\}$ and
\item[(ii)] $c_i\geq 2$ for each $i\in\{r+1, \ldots, 2r\}$.
\end{itemize}
\end{Lemma}

\begin{proof}
Set $n:=|V(G)|=2r+3s$. First, suppose that (i) and (ii) holds. Then$$x_1^2x_2^2\cdots x_n^2=\prod_{i=1}^r(x_ix_r+i)^2\prod_{i=1}^s\big((x_{2r+i}x_{2r+s+i})(x_{2r+s+i}x_{2r+2s+i})(x_{2r+i}x_{2r+2s+i})\big)$$ belongs to $\Bc(G,\mathfrak{c})$. In other words, $\Bc(G,\mathfrak{c}) = \{x_1^2x_2^2\cdots x_n^2\}$. Thus, $\conv(\Dc(G,\mathfrak{c}))$ is equal to $\Qc'_n + (1,\ldots,1)$, which is Gorenstein.

Conversely, suppose that $\conv(\Dc(G,\mathfrak{c}))$ is Gorenstein. We prove (i) and (ii) hold. Set $\delta:=\delta_{\mathfrak{c}}(I(G))$. Also, for each $i=1, \ldots, n$ set $\rho_i:=\rho_{(G,\mathfrak{c})}$. By Lemma \ref{noinclusion1}, for each $i\in [n]$ with $i\notin \{2r+1, \ldots, 2r+s\}$, the singleton $\{i\}$ is $\rho_{(G,\mathfrak{c})}$-closed. So, we have the following cases.

\medskip

{\bf Case 1.} Suppose that for each $i\in \{2r+1, \ldots, 2r+s\}$, the singleton $\{i\}$ is $\rho_{(G,\mathfrak{c})}$-closed. This implies that for each $i\in [n]$, the the singleton $\{i\}$ is $\rho_{(G,\mathfrak{c})}$-closed. Obviously, every singleton is $\rho_{(G,\mathfrak{c})}$-inseparable too. Thus, we conclude from Lemma \ref{criterion} that either, $\rho_1=\cdots = \rho_n = 1$ or $\rho_1=\cdots = \rho_n = 2$. In the first case, it follows from Lemma \ref{cmcw1} that $\conv(\Dc(G,\mathfrak{c}))$ is not Gorenstein. Therefore, assume that $\rho_1=\cdots = \rho_n = 2$. It follows from these equalities that $c_i\geq 2$, for each $i=1, \ldots, n$. Moreover, since for each $i\in [2r+3s]\setminus \{r+1, \ldots, 2r\}$, we have ${\rm deg}_G(x_i)\geq 2$, using Lemma \ref{rhomin}, we deduce that $c_i=2$. Thus, (i) and (ii) hold in this case.

\medskip

{\bf Case 2.} Suppose that there is an integer $i$ with $i\in \{2r+1, \ldots, 2r+s\}$ such that the singleton $\{i\}$ is not $\rho_{(G,\mathfrak{c})}$-closed. Let $A$ be a maximal subset of $[n]$ with $i\in A$ and $\rho_{(G,\mathfrak{c})}(A)=\rho_i$. In particular, $|A|\geq 2$. 

\medskip 

{\bf Claim 1.} There is a nonempty subset $B$ of $\{r+1, \ldots, 2r\}$ such that $A=B\cup\{i\}$. Moreover, if $r+t\in B$, then the vertices $x_t$ and $x_i$ are adjacent in $G$.

\medskip 

{\it Proof of Claim 1.} Let $v\in \Bc(G,\mathfrak{c})$ be a monomial with ${\rm deg}_{x_i}(v)=\rho_i$. Then $v$ can be written as $v=f_1\cdots f_{\delta}$, where $f_1, \ldots, f_{\delta}$ are edges of $G$. 

We show that every integer $j$ with $j\notin \{r+1, \ldots, 2r\}\cup\{i\}$ does not belong to $A$. Indeed, if $x_j$ divides $v$, then$$\rho_{(G,\mathfrak{c})}(\{i,j\}) \geq {\rm deg}_{x_i}(v)+{\rm deg}_{x_j}(v)>{\rm deg}_{x_i}(v)=\rho_i.$$So, in this case, $j\notin A$. Assume that $x_j$ does not divide $v$. Since $j\notin \{r+1, \ldots, 2r\}$, it follows from the structure of $G$ that there is a vertex $x_{\ell}\in N_G(x_j)\setminus N_G(x_i)$. If $x_{\ell}$ does not divide $v$, then $(x_jx_{\ell})v\in (I(G)^{\delta+1})_{\mathfrak{c}}$ which is a contradiction. Therefore, $x_{\ell}$ divides $v$. Hence, in the representation of $v$ as $v=f_1\cdots f_{\delta}$, there is an edge, say $f_1$ which is incident to $x_{\ell}$. In other words, $f_1=\{x_{\ell}, x_{\ell'}\}$ for a vertex $x_{\ell'}\in V(G)$. Since $x_{\ell}\notin N_G(x_i)$, one has $x_{\ell'}\neq x_i$. Then$$\frac{vx_j}{x_{\ell'}}=(x_jx_{\ell})f_2\cdots f_{\delta}\in \Bc(G,\mathfrak{c}).$$This yields that
\begin{align*}
\rho_{(G,\mathfrak{c})}(\{i, j\})& \geq {\rm deg}_{x_i}(vx_j/x_{\ell'})+{\rm deg}_{x_j}(vx_j/x_{\ell'})\\ &>{\rm deg}_{x_i}(vx_j/x_{\ell'})={\rm deg}_{x_i}(v)=\rho_i.
\end{align*}
Hence, $j\notin A$. Consequently, there is a subset $B$ of $\{r+1, \ldots, 2r\}$ such that $A=B\cup\{i\}$. Since $|A|\geq 2$, we deduce the $B$ is nonempty. The same argument as above shows that if $r+t\in B$, then $N_G(x_{r+t})\subseteq N_G(x_{i})$. In other words, $x_t$ and $x_i$ are adjacent in $G$. This proves Claim 1.

\medskip

{\bf Claim 2.} $\rho_{(G,\mathfrak{c})}(A)\geq c_{i+s}+c_{i+2s}+\sum_{r+k\in B}c_k$.

\medskip

{\it Proof of Claim 2.} Let $v$ be the monomial defined in the proof of Claim 1. Assume that $r+k\in B\subset A$. Since $\rho_{(G,\mathfrak{c})}(\{i, r+k\})=\rho_i$, it follows that $v$ is not divisible by $x_{r+k}$. If ${\rm deg}_{x_k}(v)< c_k$, then  $(x_kx_{r+k})v\in (I(G)^{\delta+1})_{\mathfrak{c}}$ which is a contradiction. Thus, ${\rm deg}_{x_k}(v)= c_k$, for each integer $k$ with $r+k\in B$. Assume that in the representation of $v$ as $v=f_1\cdots f_{\delta}$, there is an edge, say $f_{\delta}$ which is incident to $x_k$ but not to $x_i$. Then $f_{\delta}=\{x_k,x_{k'}\}$, for some vertex $x_{k'}\neq x_i$. This yields that$$\frac{vx_{r+k}}{x_{k'}}=f_1\cdots f_{\delta-1}(x_kx_{r+k})\in \Bc(G,\mathfrak{c}).$$Thus,
\begin{align*}
\rho_{(G,\mathfrak{c})}(\{i, r+k\})& \geq {\rm deg}_{x_i}(vx_{r+k}/x_{k'})+{\rm deg}_{x_{r+k}}(vx_{r+k}/x_{k'})\\ &>{\rm deg}_{x_i}(vx_{r+k}/x_{k'})={\rm deg}_{x_i}(v)=\rho_i,
\end{align*}
which is a contradiction as $\rho_{(G,\mathfrak{c})}(\{i, r+k\})=\rho_i$. This contradiction shows that in the representation of $v$ as $v=f_1\cdots f_{\delta}$, if an edge $f_{\ell}$ is incident to $x_k$, it is incident to $x_i$ too. Suppose that ${\rm deg}_{x_{i+s}}(v)< c_{i+s}$. Then for each integer $k$ with $r+k\in B$, one has $$(x_{i+s}x_{r+k})v=(x_ix_{i+s})(x_kx_{r+k})v/(x_kx_i)\in (I(G)^{\delta+1})_{\mathfrak{c}},$$a contradiction. Therefore, ${\rm deg}_{x_{i+s}}(v)= c_{i+s}$. By symmetry, ${\rm deg}_{x_{i+2s}}(v)= c_{i+2s}$. If in the representation of $v$ as $v=f_1\cdots f_{\delta}$, there is an edge which is equal to $\{x_{i+s}, x_{i+2s}\}$, then for any integer $k$ with $r+k\in B$, one has $$\frac{vx_{r+k}}{x_{i+2s}}=\frac{(x_kx_{r+k})(x_ix_{i+s})v}{(x_ix_k)(x_{i+s} x_{i+2s})}\in \Bc(G,\mathfrak{c}).$$Therefore,
\begin{align*}
\rho_{(G,\mathfrak{c})}(\{i, r+k\})& \geq {\rm deg}_{x_i}(vx_{r+k}/x_{i+2s})+{\rm deg}_{x_{r+k}}(vx_{r+k}/x_{i+2s})\\ &>{\rm deg}_{x_i}(vx_{r+k}/x_{i+2s})={\rm deg}_{x_i}(v)=\rho_i,
\end{align*}
which is a contradiction. Hence, the edge $\{x_{i+s}, x_{i+2s}\}$ does not appear in the representation of $v$. In other words, in the representation of $v$, any edge incident to $x_{i+s}$ (resp. $x_{i+2s}$) is $\{x_i, x_{i+s}\}$ (resp. $\{x_i, x_{i+2s}\}$). Consequently,
\begin{align*}
\rho_{(G,\mathfrak{c})}(A)&=\rho_i={\rm deg}_{x_i}(v)\geq {\rm deg}_{x_{s+i}}(v)+{\rm deg}_{x_{2s+i}}(v)+\sum_{r+k\in B}{\rm deg}_{x_k}(v)\\ & =c_{i+s}+c_{i+2s}+\sum_{r+k\in B}c_k.   
\end{align*}
This proves Claim 2.

\medskip

We show that $A$ is $\rho_{(G,\mathfrak{c})}$-inseparable. Indeed, assume that $A_1$ and $A_2$ are proper disjoint subsets of $A$ with $A_1\cup A_2=A$. We may assume that $x_i\in A_1$. Then $\rho_{(G,\mathfrak{c})}(A_1)=\rho_i$. Hence, $\rho_{(G,\mathfrak{c})}(A_1)+\rho_{(G,\mathfrak{c})}(A_2)> \rho_{(G,\mathfrak{c})}(A)$. Thus, $A$ is a $\rho_{(G,\mathfrak{c})}$-inseparable subset of $A$. Since $A$ is a maximal subset of $[n]$ with $\rho_{(G,\mathfrak{c})}(A)=\rho_i$, we conclude that $A$ is $\rho_{(G,\mathfrak{c})}$-closed. By Lemma \ref{noinclusion1}, the singletons $\{i+s\}$ is a $\rho_{(G,\mathfrak{c})}$-closed and $\rho_{(G,\mathfrak{c})}$-inseparable subset of $[n]$. Therefore, by Lemma \ref{criterion}, one has either $\rho_{i+s}=1$ or $\rho_{i+s}=2$. However, $\rho_{i+s}=1$ is not possible, as $A$ is $\rho_{(G,\mathfrak{c})}$-closed and $\rho_{(G,\mathfrak{c})}$-inseparable with $\rho_{(G,\mathfrak{c})}(A)\geq |A|+1$ (Claim 2). So, suppose that $\rho_{i+s}=2$.  It then follows from Lemma \ref{criterion} that $\rho_{(G,\mathfrak{c})}=|A|+1$. However, since $c_{i+s}\geq \rho_{i+s}=2$, we deduce from Claim 2 that $\rho_{(G,\mathfrak{c})}\geq |A|+2$, which is a contradiction. 
\end{proof}

The following theorem is an immediate consequence of Lemma \ref{cmcw} and its proof.

\begin{Theorem}
    \label{thm:CMCW}  
The Gorenstein polytopes of the form $\conv(\Dc(G),\mathfrak{c}))$, where $G$ is a Cohen--Macaulay Cameron--Walker graph on $n$ vertices and where $\mathfrak{c}\in (\ZZ_{>0})^{n}$, are exactly ${\Qc}'_{2n} + (1,\ldots, 1)$.     
\end{Theorem}

\section*{Acknowledgments}
The second author is supported by a FAPA grant from Universidad de los Andes.

\section*{Statements and Declarations}
The authors have no Conflict of interest to declare that are relevant to the content of this article.

\section*{Data availability}
Data sharing does not apply to this article as no new data were created or analyzed in this study.

\end{document}